\documentclass[reqno]{amsart}

\usepackage{xypic}
\input xy
\xyoption{all}
\usepackage{epsfig}
\usepackage{color}
\usepackage{amsthm}
\usepackage{amssymb}
\usepackage{amsmath}
\usepackage{amscd}
\usepackage{amsopn}
\usepackage{url}

\usepackage{color} %\textcolor{red}{Test}

\newcommand{\labell}[1] {\label{#1}}
\numberwithin{equation}{section}
\newtheorem {Theorem}{Theorem}
\numberwithin{Theorem}{section}

\newtheorem {Lemma}[Theorem]    {Lemma}
\newtheorem {Claim}[Theorem]    {Claim}

\newtheorem {Proposition}[Theorem]{Proposition}
\newtheorem {Corollary}[Theorem]{Corollary}
\theoremstyle{definition}
\newtheorem{Definition}[Theorem]{Definition}
\theoremstyle{remark}
\newtheorem{Remark}[Theorem]{Remark}
\newtheorem{Example}[Theorem]{Example}

\expandafter\chardef\csname pre amssym.def at\endcsname=\the\catcode`\@
\catcode`\@=11
\def\undefine#1{\let#1\undefined}
\def\newsymbol#1#2#3#4#5{\let\next@\relax
 \ifnum#2=\@ne\let\next@\msafam@\else
 \ifnum#2=\tw@\let\next@\msbfam@\fi\fi
 \mathchardef#1="#3\next@#4#5}
\def\mathhexbox@#1#2#3{\relax
 \ifmmode\mathpalette{}{\m@th\mathchar"#1#2#3}%
 \else\leavevmode\hbox{$\m@th\mathchar"#1#2#3$}\fi}
\def\hexnumber@#1{\ifcase#1 0\or 1\or 2\or 3\or 4\or 5\or 6\or 7\or 8\or
 9\or A\or B\or C\or D\or E\or F\fi}

\font\teneufm=eufm10
\font\seveneufm=eufm7
\font\fiveeufm=eufm5
\newfam\eufmfam
\textfont\eufmfam=\teneufm
\scriptfont\eufmfam=\seveneufm
\scriptscriptfont\eufmfam=\fiveeufm

\catcode`\@=\csname pre amssym.def at\endcsname

\def    \eps    {\epsilon}

\newcommand{\CS}{{\mathcal S}}

\newcommand{\supp}{\operatorname{supp}}

\newcommand{\id}{{\mathit id}}
\newcommand{\const}{{\mathit const}}

\newcommand{\tf}{\tilde{f}}
\newcommand{\tgamma}{\tilde{\gamma}}

\newcommand{\tH}{\tilde{H}}

\newcommand{\Ii}{{\mathcal I}}

\def    \C      {{\mathbb C}}
\def    \R      {{\mathbb R}}

\def    \Z      {{\mathbb Z}}

\def    \12    {{\frac{1}{2}}}

\def    \rk     {\operatorname{rk}}

\def    \Sp     {\operatorname{Sp}}
\def    \U     {\operatorname{U}}

\def    \HF     {\operatorname{HF}}

\def    \HM     {\operatorname{HM}}

\def    \MUCZ  {\operatorname{\mu_{\scriptscriptstyle{CZ}}}}

\begin{document}

%%%%%%%%%%%%%%%%%%%%%%%%%%%%%%
%   TEXT FORMATTING

\setlength{\smallskipamount}{6pt}
\setlength{\medskipamount}{10pt}
\setlength{\bigskipamount}{16pt}

%%%%%%%%%%%%%%%%%%%%%%%%%%

%%%%%%%%%%%%%%%%%%%%%%%%%%

%%%%%%%%%%%           BEGINNING OF  TEXT

%%%%%%%%%%%%%%%%%%%%%%%%%%

\title[Local Floer Homology and the Action Gap]
{Local Floer Homology and the Action Gap}

\author[Viktor Ginzburg]{Viktor L. Ginzburg}
\author[Ba\c sak G\"urel]{Ba\c sak Z. G\"urel}

\address{VG: Department of Mathematics, UC Santa Cruz,
Santa Cruz, CA 95064, USA}
\email{ginzburg@math.ucsc.edu}

\address{BG: Centre de recherches math\'ematiques,
Universit\'e de Montr\'eal,
P.O. Box 6128, Centre--ville Station,
Montr\'eal, QC, H3C 3J7, CANADA}
\email{gurel@crm.umontreal.ca}

\subjclass[2000]{53D40, 37J45, 70H12}
%\keywords{Periodic orbits, Hamiltonian flows, Floer homology, Hofer--Zehnder
%capacity}
\date{\today}
\thanks{The work is partially supported by the NSF and by the faculty
research funds of the University of California, Santa Cruz.}

%\bigskip

\begin{abstract}
In this paper, we study the behavior of the local Floer homology of an
isolated fixed point and the growth of the action gap under
iterations. To be more specific, we prove that an isolated fixed point
of a Hamiltonian diffeomorphism remains isolated for a certain class
of iterations (the so-called admissible iterations) and that the local
Floer homology groups for all such iterations are isomorphic to each
other up to a shift of degree.  Furthermore, we study the
pair-of-pants product in local Floer homology, and characterize a
particular class of isolated fixed points (the symplectically
degenerate maxima), which plays an important role in the proof of the
Conley conjecture. The proofs of these facts rely on an observation
that for a general diffeomorphism, not necessarily Hamiltonian, an
isolated fixed point remains isolated under all admissible iterations.
Finally, we apply these results to show that for a quasi-arithmetic
sequence of admissible iterations of a Hamiltonian diffeomorphism with
isolated fixed points the minimal action gap is bounded from above
when the ambient manifold is closed and symplectically aspherical.
This theorem is a generalization of the Conley conjecture.

\end{abstract}

\maketitle

\section{Introduction}
\labell{sec:intro}

In this paper, we analyze the behavior of the local Floer homology of
an isolated fixed point under iterations of a Hamiltonian
diffeomorphism.  To be more precise, we prove that an isolated fixed
point remains isolated for a certain class of iterations of the
diffeomorphism (the so-called admissible iterations) and that the
local Floer homology groups for all such iterations persist, i.e.,
these groups are isomorphic to each other up to a shift of degree.
Then we use this result to study the pair-of-pants product in local
Floer homology, and characterize in homological and geometrical terms
a particular class of isolated fixed points, the so-called
symplectically degenerate maxima, that play an important role in the
proof of the Conley conjecture discussed below. The proofs of these
facts rely on an observation of independent interest that for a
general diffeomorphism, not necessarily Hamiltonian, an isolated fixed
point remains isolated under admissible iterations.

As an application of the persistence of local Floer homology to global
Hamiltonian dynamics on closed, symplectically aspherical manifolds,
we show that for a quasi-arithmetic sequence of admissible iterations
of a Hamiltonian diffeomorphism with isolated fixed points the minimal
action gap is bounded from above. This theorem can be viewed as a
generalization of the Conley conjecture established in
\cite{Gi:conley,Hi} (see also \cite{FrHa,Hi,LeC}) and asserting that a
Hamiltonian diffeomorphism of a closed symplectically aspherical
manifold has simple periodic points of arbitrarily large period.

\subsection{Main results}
\labell{sec:result}

Let $M$ be a smooth manifold and let $p$ be a fixed point of a
diffeomorphism $\varphi\colon M\to M$. We call a positive integer $k$
\emph{admissible} (with respect to $p$) if $\lambda^k\neq 1$ for all
eigenvalues $\lambda\neq 1$ of $d\varphi_p\colon T_pM\to T_pM$. In
other words, $k$ is admissible if and only if $d\varphi_p^k$ and
$d\varphi_p$ have the same generalized eigenvectors with eigenvalue
one. For instance, when no eigenvalue $\lambda\neq 1$ is a root of
unity, all $k> 0$ are admissible. An iteration $k$ is called
\emph{good} with respect to $p$ if the parity of the number of pairs
$\{\lambda,\lambda^{-1}\}$ of negative real eigenvalues, counted with
multiplicity, is the same for $d\varphi_p^k$ and $d\varphi_p$. (Since
the eigenvalue $-1$ has even multiplicity, the number of
pairs is well defined.) Otherwise, $k$ is
a \emph{bad} iteration.  Furthermore, $k$ is said to be admissible
(good) for $\varphi$ if it is admissible (good) with respect to every
fixed point of $\varphi$.

Assume now that $M$ is symplectic and $\varphi=\varphi_H$ is the
time-one map of the Hamiltonian time-dependent flow $\varphi_H^t$
generated by a one-periodic in time Hamiltonian $H$. Let $\gamma$ be a
one-periodic orbit of $H$. Then $k$ is (good) admissible with respect
to $\gamma$ if it is (good) admissible for $\varphi_H$ with respect to
the fixed point $p=\gamma(0)$.  In what follows, we denote by $H^{\#
  k}$ and $\gamma^k$ the $k$th iterations of $H$ and, respectively,
$\gamma$. (By definition, the Hamiltonian $H^{\# k}$ generates the
time-dependent flow $(\varphi_H^t)^k$ and
$\gamma^k(t)=\varphi_{H^{\#k}}^t(p)$, where $p=\gamma(0)$ and $t\in
[0,\, 1]$.  One can think of $\gamma^k$ as the $k$-periodic orbit
$\gamma(t)$, $t\in [0,\, k]$, of $H$; see Section
\ref{sec:notations}.)  When $M$ is symplectically aspherical, the
local Floer homology groups $\HF_*(H,\gamma)$ are defined similarly to
ordinary Floer homology, but taking into account only one-periodic
orbits that $\gamma$ splits into under a non-degenerate perturbation
of $H$.  For instance, when $\gamma$ is non-degenerate,
$\HF_l(H,\gamma)$ is $\Z_2$ if $l$ is equal to the Conley--Zehnder
index of $\gamma$ and is zero otherwise. Let also $\Delta_H(\gamma)$
and $A_H(\gamma)$ denote the mean index of $\gamma$ and, respectively,
the action of $H$ on $\gamma$. We refer the reader to Sections
\ref{sec:mean-index} and \ref{sec:LFH} for a detailed discussion of
the mean index and local Floer homology, and for further references.

The first result of this paper asserts that up to a shift of degree
$s_k$ the local Floer homology groups $\HF_*(H^{\# k},\gamma^k)$ are
isomorphic for all admissible iterations $\varphi_H^k$ and, moreover,
the mean shift $s_k/k$ converges to the mean index $\Delta_H(\gamma)$.
Thus, the general behavior of the local Floer homology groups under
iterations is similar to that for non-degenerate periodic orbits.

\begin{Theorem}
\labell{thm:persist-lf}

Let $M$ be a symplectically aspherical manifold and let $\gamma$ be an
isolated one-periodic orbit of a Hamiltonian $H\colon S^1\times M\to
\R$. Then $\gamma^k$ is also an isolated one-periodic orbit of $H^{\#
k}$ for all admissible $k$ and the local Floer homology groups of $H$
and $H^{\# k}$ coincide up to a shift of degree:
\begin{equation}
\labell{eq:case2}
\HF_*(H^{\# k} ,\gamma^k)=\HF_{*+s_k}(H,\gamma)
\quad\text{for some $s_k$.}
\end{equation}
Furthermore, $\lim_{k\to \infty}s_k/k=\Delta_H(\gamma)$ and the shift
$s_k$ is even if $k$ is good, provided that $\HF_*(H,\gamma)\neq 0$
and hence the shifts $s_k$ are uniquely determined by
\eqref{eq:case2}.  Moreover, when $\Delta_H(\gamma)=0$ and
$\HF_n(H,\gamma)\neq 0$, the orbit $\gamma$ is strongly degenerate
(see Section \ref{sec:notations}) and all $s_k$ are zero.

\end{Theorem}

\begin{Remark}
\labell{rmk:persist-lf} The assumption that $M$ is symplectically
aspherical is not essential in Theorem \ref{thm:persist-lf}. The local
Floer homology groups $\HF_l(H,\gamma)$ are defined for an arbitrary
symplectic manifold $M$, although in this case the orbit $\gamma$ must
be equipped with a capping. With this modification, Theorem
\ref{thm:persist-lf}, being a local result, holds for any symplectic
manifold.
\end{Remark}

One ingredient of the proof of Theorem \ref{thm:persist-lf} concerns
``persistence of isolation'' for fixed points of smooth, not
necessarily Hamiltonian, diffeomorphisms and is of independent
interest. This is the following result proved in
Section~\ref{sec:persist-iso}:

\begin{Proposition}
\labell{prop:persist-iso}

Let $p\in M$ be an isolated fixed point of a $C^1$-smooth
diffeomorphism $\varphi\colon M\to M$. Then $p$ is also an isolated
fixed point of $\varphi^k$ for every admissible~$k$.
\end{Proposition}

\begin{Remark}
\labell{rmk:persist-iso}
In the proof of Theorem \ref{thm:persist-lf}, we will also make use of
a parametric version of Proposition \ref{prop:persist-iso}. Namely,
\emph{assume that $p\in M$ is a uniformly isolated fixed point of a
family of $C^1$-smooth diffeomorphisms $\varphi_s\colon M\to M$ with
$s\in [0,\,1]$, i.e., $p$ is the only fixed point of $\varphi_s$,
for all $s$, in some (independent of $s$) neighborhood of $p$. Then
$p$ is also a uniformly isolated fixed point of $\varphi^k_s$ for
every $k$ which is admissible for all $\varphi_s$.} The proof of
this generalization of Proposition \ref{prop:persist-iso} is given in
Section~\ref{sec:persist-iso}.
\end{Remark}

One-periodic orbits $\gamma$ with $\Delta_H(\gamma)=0$ and
$\HF_n(H,\gamma)\neq 0$ arise naturally in the proof of the Conley
conjecture (see \cite{Gi:conley,Hi}) and are referred to here as
\emph{symplectically degenerate maxima}. Utilizing Theorem
\ref{thm:persist-lf} and the results from \cite{Gi:conley}, we give
homological and geometrical characterizations of symplectically
degenerate maxima in Section \ref{sec:sdm}. Furthermore, we show that
the pair-of-pants product in $\HF_*(H,\gamma)$ has strong ``vanishing
properties'' detecting, in particular, symplectically degenerate
maxima. Namely, the product is nilpotent if and only if $\gamma$ is
not a symplectically degenerate maximum; see Section
\ref{sec:product}.

It is natural to consider Theorem \ref{thm:persist-lf} in the context
of the Shub--Sullivan theorem asserting that whenever $p$ is an
isolated fixed point for all iterations $\varphi^k$ of a $C^1$-smooth
map $\varphi$ (which is not required to be a diffeomorphism), the
index of $\varphi^k$ at $p$ is bounded; see \cite{SS}. In the setting
of Theorem \ref{thm:persist-lf}, the index of $\varphi_H^k$ at
$\gamma(0)$ is equal to $\sum_l (-1)^l \dim_{\Z_2}\HF_l(H^{\# k}
,\gamma^k)$. Thus, the absolute value of the index is independent of
$k$ and the index is bounded, as long as $k$ is admissible.
(However, this consequence of Theorem \ref{thm:persist-lf} can also be
extracted from the proofs of the Shub--Sullivan theorem and of
Proposition \ref{prop:persist-iso}, and hence holds in much greater
generality, cf.\ Remark \ref{rmk:SS}.) Using Theorem
\ref{thm:persist-lf}, it is easy to prove the following literal,
Hamiltonian analogue of the Shub--Sullivan theorem:

\begin{Corollary}
Let $\gamma$ be a one-periodic orbit of a Hamiltonian $H$ on a
symplectically aspherical manifold. Assume that the orbit $\gamma^k$
is isolated for all $k>0$. Then $\rk\HF_*(H^{\#
k},\gamma^k):=\sum_l\dim_{\Z_2}\HF_l(H^{\# k},\gamma^k)$ is bounded as
a function of $k$.
\end{Corollary}

\begin{Remark}
  In fact, in this corollary and in the Shub--Sullivan theorem, it is
  sufficient to assume that $\gamma^k$ is isolated only for a certain
  finite collection of $k$. (This is a consequence of
  Proposition \ref{prop:persist-iso}.)  For instance, if no eigenvalue
  $\lambda\neq 1$ of $d\varphi_H$ is a root of unity, it suffices to
  require $\gamma$ to be isolated.
\end{Remark}

The analogy with the Shub--Sullivan theorem and with the results of
Gromoll and Meyer, \cite{GM2}, suggests a number of applications of
Theorem \ref{thm:persist-lf} to the existence problem for periodic
points of Hamiltonian diffeomorphisms. Namely, for some Hamiltonian
diffeomorphisms of non-compact manifolds or symplectomorphisms arising
in classical Hamiltonian dynamics, the rank of (filtered) Floer
homology appears to grow with the order of iteration, and then
the Hamiltonian Shub--Sullivan theorem implies the existence of
infinitely many periodic orbits.  Here, leaving these applications
aside, we focus on just one general result concerning the behavior of
the action spectrum of $H^{\# k}$. To state this result, let us recall
one more definition.

A strictly increasing, infinite sequence of positive integers
$$
\nu_1<\nu_2<\nu_3<\ldots
$$
is called \emph{quasi-arithmetic} if $\nu_{i+1}-\nu_i<\const$ for all
$i$ and some constant independent of $i$. For example, any set
containing an infinite arithmetic progression is
quasi-arithmetic. Furthermore, it is easy to see that whenever fixed
points of $\varphi$ are isolated, the set of (good) admissible
iterations is quasi-arithmetic.  (Indeed, the set of admissible
iterations is comprised of integers that are not divisible by the
degrees $q_1>1,\ldots, q_r>1$ of the roots of unity among the
eigenvalues $\lambda\neq 1$ of $d\varphi$ at the fixed points of
$\varphi$. This set contains the arithmetic sequence
$m_k=1+q_1\cdot\ldots\cdot q_r\cdot k$. To ensure that the iterations
are good, it suffices to add $q_0=2$ to the collection of
$q_1,\ldots,q_r$.)

\begin{Theorem}
\labell{thm:main}

Let $H\colon S^1\times M\to \R$ be a Hamiltonian on a closed
symplectically aspherical manifold $M$ such that all fixed points of
$\varphi_H$ are isolated.  Then there exist an infinite quasi-arithmetic
sequence ${\nu_i}$ of admissible iterations of $\varphi_H$, a sequence
$y_i$ of $\nu_i$-periodic orbits of $H$, and a one-periodic orbit $x$
of $H$ such that
\begin{itemize}
\item $|A_{H^{\#\nu_i}}(x^{\nu_i})-A_{H^{\#\nu_i}}(y_i)|\leq e$,
\item $|\Delta_{H^{\# \nu_i}}(x^{\nu_i})-\Delta_{H^{\#
\nu_i}}(y_i)|\leq \delta$,
\item $|A_{H^{\#\nu_i}}(x^{\nu_i})-A_{H^{\#\nu_i}}(y_i)|+
|\Delta_{H^{\# \nu_i}}(x^{\nu_i})-\Delta_{H^{\# \nu_i}}(y_i)|>0$
\end{itemize}
for some constants $e$ and $\delta$ independent of $i$. Furthermore,
any infinite quasi-arithmetic sequence of admissible iterations contains a
quasi-arithmetic subsequence $\nu_i$ with these properties.
\end{Theorem}

\begin{Remark}
  Under suitable additional assumptions on $M$ and/or $H$, Theorem
  \ref{thm:main} extends to closed, weakly monotone symplectic
  manifolds. However, this generalization is far less obvious than the
  generalization of Theorem \ref{thm:persist-lf} mentioned in Remark
  \ref{rmk:persist-lf} and it will be discussed elsewhere.  Also note
  that, as simple examples show, the condition that the fixed points
  of $\varphi_H$ are isolated is essential in Theorem~\ref{thm:main}.
\end{Remark}

Theorem \ref{thm:main} can be readily interpreted as a statement about
the behavior of the action and index gaps for the iterations
$H^{\#\nu_i}$. An \emph{action gap} of $H$ is the difference
$|A_H(\gamma_1)-A_H(\gamma_0)|$ for two distinct one-periodic orbits
$\gamma_0$ and $\gamma_1$ of $H$. An \emph{index gap} is defined in a
similar fashion by using the mean index $\Delta_H(\gamma)$ in place of
$A_H(\gamma)$ and an \emph{action--index gap} is the sum
$\Gamma_H(\gamma_1,\gamma_0):=|A_H(\gamma_1)-A_H(\gamma_0)|+
|\Delta_H(\gamma_1)-\Delta_H(\gamma_0)|$.  The connection between the
Conley conjecture and the growth of action gaps under iterations can
be summarized as the fact that if the Conley conjecture failed to
hold, the minimal non-zero action gap would grow linearly
with the order of iteration. For instance, the proofs of various
versions of the Conley conjecture for Hamiltonians with displaceable
support are based on the observation that in this case a certain
positive action gap of $H^{\# k}$ remains bounded from above as
$k\to\infty$; see \cite{FS,Gu,HZ,schwarz,Vi:gen}. (For such
Hamiltonians, the Conley conjecture asserts the existence of simple
periodic points with non-zero action and arbitrarily large period,
provided that $\varphi_H\neq \id$.)  Yet, although Theorem
\ref{thm:main} does ensure that certain action gaps remain bounded, it
does not guarantee that these gaps are non-zero. This difficulty is
overcome once action gaps are replaced by action--index gaps, and hence,
Theorem \ref{thm:main} still implies the Conley conjecture.

\begin{Corollary}[\cite{Gi:conley}]
\labell{cor:conley}

Let $\varphi\colon M\to M$ be a Hamiltonian diffeomorphism of a closed
symplectically aspherical manifold $M$. Assume that the fixed points
of $\varphi$ are isolated.  Then $\varphi$ has simple periodic orbits
of arbitrarily large period.

\end{Corollary}

\begin{proof}
Assume the contrary: $\varphi$ has only finitely many simple periodic
orbits. Let $p_1>1,\ldots, p_l>1$ be the periods of these orbits.  As
above, denote by $q_1>1,\ldots, q_r>1$ the degrees of the roots of
unity among the eigenvalues $\lambda\neq 1$ of $d\varphi$ at the fixed
points of $\varphi$. The integers not divisible by
$p_1,\ldots,p_l,q_1,\ldots,q_r$ are admissible and form a
quasi-arithmetic sequence. Pick a sequence of iterations $\nu_i$
contained in this set such that $0<\Gamma(x^{\nu_i},y_i)<c:=e+\delta$
as in Theorem \ref{thm:main}. By our choice of $\nu_i$, every
$\nu_i$-periodic orbit is necessarily the $\nu_i$th iteration of a
one-periodic orbit and, in particular, $y_i=z_i^{\nu_i}$ for some
one-periodic orbits $z_i$. As a consequence,
$\Gamma(x^{\nu_i},y_i)=\nu_i\Gamma(x,z_i)$ and
$\Gamma(x,z_i)>0$. Denote by $\eps>0$ the minimal positive
action-index gap between one-periodic orbits of $\varphi$. Then
$\Gamma(x,z_i)\geq\eps$ and $\Gamma(x^{\nu_i},y_i)\geq \nu_i\eps>c$,
when $\nu_i$ is large enough. This contradicts Theorem~\ref{thm:main}.
\end{proof}

\begin{Remark}
  Let $\varphi$ be a compactly supported, positive Hamiltonian
  diffeomorphism of $\R^{2n}$, i.e., $\varphi=\varphi_H$, where $H$ is
  compactly supported, $H\geq 0$, and $H\not\equiv 0$. Then the number
  of simple periodic orbits of $\varphi$ with positive action and with
  period less than or equal to $k$ grows at least linearly with $k$,
  \cite{Vi:gen}. Moreover, the same is true for any positive
  Hamiltonian diffeomorphism $\varphi$ of a wide, geometrically
  bounded manifold (e.g., a manifold convex at infinity) whenever
  $\varphi$ has compact displaceable support, \cite{Gu}. To the best
  of the authors' knowledge, no such growth results have been obtained
  yet either without the positivity assumption or for diffeomorphisms
  of closed manifolds.
\end{Remark}

\subsection{Organization of the paper}
In Section \ref{sec:prelim}, we set conventions and notation, recall
the definition and relevant properties of the mean index, and provide
some basic references for the construction of Floer homology. Local
Floer homology is discussed in Section \ref{sec:LFH}.
Theorem \ref{thm:persist-lf} is proved in Section
\ref{sec:persist}. The questions of homological and geometrical
characterization of symplectically degenerate maxima and of vanishing
of the pair-of-pants product in local Floer homology are addressed in
Section \ref{sec:sdm-product}.  Theorem \ref{thm:main} is proved in
Section \ref{sec:pf-main}.  The paper is concluded by a proof of
Proposition \ref{prop:persist-iso}, given in Section
\ref{sec:persist-iso} which is independent of the rest of the paper.

\section{Preliminaries}
\labell{sec:prelim}

In this section, we set notation and conventions used in the paper,
recall relevant facts concerning Floer homology and the mean index,
and provide necessary references for the definitions and proofs.

\subsection{Conventions and notation}
\labell{sec:notations}

Throughout the paper, $(M,\omega)$ denotes a symplectic manifold of
dimension $2n$ or, sometimes, $M$ is just a smooth, $m$-dimensional
manifold.  When $M$ is symplectic, it is always required to be
symplectically aspherical, i.e.,
$\omega\mid_{\pi_2(M)}=0=c_1(TM)\mid_{\pi_2(M)}$, although in some instances
(e.g., Theorem \ref{thm:persist-lf}) this requirement can be
relaxed. All maps and functions considered in this paper are assumed
to be $C^\infty$-smooth and all Hamiltonians $H$ are one-periodic in
time, i.e., $H\colon S^1\times M\to\R$, unless specified otherwise.
We set $H_t = H(t,\cdot)$ for $t\in S^1$. The Hamiltonian vector field
$X_H$ of $H$ is defined by $i_{X_H}\omega=-dH$. The time-dependent
Hamiltonian flow of $H$, i.e., the flow of $X_H$, is denoted by
$\varphi_H^t$.  (By definition, a (time-dependent) flow is just
a family of diffeomorphisms beginning at $\id$.) We refer to the
time-one map $\varphi_H^1=:\varphi_H$ as a Hamiltonian diffeomorphism.
One- or $k$-periodic orbits of $\varphi_H^t$ are in one-to-one
correspondence with fixed points or $k$-periodic points of
$\varphi_H$. In this paper, we are only concerned with contractible
periodic orbits.  \emph{A periodic orbit is always assumed to be
  contractible, even if this is not explicitly stated}.

Let $\gamma\colon S^1\to M$ be a contractible loop. The action of $H$
on $\gamma$ is given by
$$
A_H(\gamma)=-\int_z\omega+\int_{S^1} H_t(\gamma(t))\,dt,
$$
where $z\colon D^2\to M$ is such that $z\mid_{S^1}=\gamma$.  The least
action principle asserts that the critical points of $A_H$ on the
space of all contractible maps $\gamma\colon S^1\to M$ are exactly the
contractible one-periodic orbits of $\varphi_H^t$.

The action spectrum $\CS(H)$ of $H$ is the set of critical values of
$A_H$. This is a zero measure, closed set; see, e.g.,
\cite{HZ,schwarz}. The \emph{index spectrum} of $H$ is defined in a
similar fashion by using the mean index $\Delta_H(\gamma)$ in place of
$A_H(\gamma)$. (The definition and properties of the mean index are
reviewed in Section \ref{sec:mean-index}.) The index spectrum
$\CS_{\Ii}(H)$ is a closed set. However, $\CS_{\Ii}(H)$, in contrast
with $\CS(H)$, need not have zero measure. The \emph{action--index
spectrum} of $H$ is the collection of pairs
$(A_H(\gamma),\Delta_H(\gamma))\in \R^2$ for all contractible
one-periodic orbits $\gamma$ of $H$; cf.\ \cite{CFHW}. This is a
closed, zero measure subset of $\R^2$.  Clearly, a non-zero action
(index) gap introduced in Section \ref{sec:result} is the distance
between two points in $\CS(H)$ (respectively, $\CS_{\Ii}(H)$).

\begin{Definition}
A fixed point $p$ of $\varphi_H$ and the one-periodic orbit
$\gamma(t)=\varphi^t_H(p)$, $t\in [0,\,1]$, are \emph{non-degenerate}
if the linearized return map $d\varphi_H \colon T_{p}M\to T_{p}M$ has
no eigenvalues equal to one. Following \cite{SZ}, we call $p$ and
$\gamma$ \emph{weakly non-degenerate} if at least one of the
eigenvalues is different from one. Otherwise,
$p$ and $\gamma$ are said to be \emph{strongly degenerate}.
\end{Definition}

The \emph{Conley--Zehnder index} of a non-degenerate periodic orbit is
defined in \cite{Sa,SZ}. In this paper, the Conley--Zehnder index
$\MUCZ(H,\gamma)\in\Z$ of an orbit $\gamma$ is set to be the negative
of that in \cite{Sa}. In other words, we normalize $\MUCZ$ so that
$\MUCZ(\gamma,H)=n$ when $\gamma$ is a non-degenerate maximum of an
autonomous Hamiltonian $H$ with small Hessian. More generally, when
$H$ is autonomous and $\gamma$ is a non-degenerate critical point of
$H$ such that the eigenvalues of the Hessian (with respect to a metric
compatible with $\omega$) are less than $2\pi$, the Conley--Zehnder
index of $\gamma$ is equal to one half of the negative signature of
the Hessian. When $H$ is clear from the context, we will use the
notation $\MUCZ(\gamma)$.

Furthermore, recall that $\pi_1(\Sp(2n))\cong\Z$, where $\Sp(2n)$ is
the group of linear symplectic transformations of $\R^{2n}=\C^n$. We
fix this isomorphism by requiring it to be the composition of the
isomorphism $\pi_1(\Sp(2n))\cong\pi_1(\U(n))$ induced by the inclusion
$\U(n)\hookrightarrow \Sp(2n)$ with the isomorphism
$\pi_1(\U(n))\cong\Z$ induced by $\det\colon \U(n)\to S^1$.  The
\emph{Maslov index} of a loop in $\Sp(2n)$ is the class of this loop
in $\pi_1(\Sp(2n))\cong \Z$. As is well known, these definitions carry
over unambiguously to the group of linear symplectic transformations
of any finite--dimensional symplectic vector space.

Let $K$ and $H$ be two one-periodic Hamiltonians. The composition $K\#
H$ is defined by the formula
$$
(K\# H)_t=K_t+H_t\circ(\varphi^t_K)^{-1}.
$$
The flow of $K\# H$ is $\varphi^t_K\circ \varphi^t_H$. In general,
$K\# H$ is not one-periodic. However, this is the case if, for
example, $H_0\equiv 0\equiv H_1$. The latter condition can be met by
reparametrizing the Hamiltonian as a function of time without changing
the time-one map. Thus, in what follows, we will always treat $K\# H$
as a one-periodic Hamiltonian.  (Another instance when the composition
$K\# H$ of two one-periodic Hamiltonians is automatically one-periodic
is when the flow $\varphi^t_K$ is a loop of Hamiltonian
diffeomorphisms, i.e., $\varphi_K^1=\id$.)  We set $H^{\# k}=H\#\ldots
\# H$ ($k$ times). The flow $\varphi^t_{H^{\# k}}=(\varphi_H^t)^k$,
$t\in [0,\,1]$, is homotopic with fixed end-points to the flow
$\varphi^t_H$, $t\in [0,\, k]$.

The $k$th iteration of a one-periodic orbit $\gamma$ of $H$ will be
denoted by $\gamma^k$. More specifically, $\gamma^k(t)=\varphi_{H^{\#
k}}^t(p)$, where $p=\gamma(0)$ and $t\in [0,\,1]$.  Clearly, $A_{H^{\#
k}}(\gamma^k)=kA_H(\gamma)$. Replacing $\varphi^t_{H^{\# k}}$, $t\in
[0,\,1]$, by the homotopic flow $\varphi^t_H$, $t\in [0,\, k]$, we can
think of $\gamma^k$ as the $k$-periodic orbit $\gamma(t)$, $t\in
[0,\,k]$, of $H$.  Hence, there is an action--preserving
one-to-one correspondence between one-periodic orbits of $H^{\# k}$
and $k$-periodic orbits of $H$.

A more detailed treatment of the material discussed in this section
can be found, for instance, in \cite{HZ}.

\subsection{Floer homology}
\labell{sec:floer}

Recall that when $M$ is closed and symplectically aspherical, the
filtered Floer homology of $H\colon S^1\times M\to \R$ for the
interval $(a,\, b)$, denoted throughout the paper by $\HF^{(a,\,
b)}_*(H)$, is defined. We refer the reader to Floer's papers
\cite{F:Morse,F:grad,F:c-l,F:witten} or to, e.g.,
\cite{BPS,HZ,mdsa,Sa,SZ,schwarz} for further references and
introductory accounts of the construction of (Hamiltonian) Floer
homology. Terminology, conventions, and most of the notation used here
are identical to those in \cite{Gi:coiso,Gi:conley,Gu}.

Consider a one-periodic Hamiltonian $G$ generating a loop of
Hamiltonian diffeomorphisms of $M$. Then, as is well known, all orbits
$\gamma(t)=\varphi_G^t(p)$, where $p\in M$ and $t\in S^1$, are
contractible loops. (This follows from the proof of the Arnold
conjecture.)  The action $c=A_G(\gamma)$ is independent of $p$ and the
Maslov index of the linearization $d(\varphi^t_G)_{\gamma(t)}$, with
respect to the trivialization of $TM$ along $\gamma$ associated with a
disk bounded by $\gamma$, is zero.  Furthermore, it is easy to see
that for a suitable choice of almost complex structures, the
composition with the flow of $G$ induces an isomorphism of Floer
complexes of $H$ and $G\# H$ shifting the action filtration by $c$ and
preserving the grading. This isomorphism sends a one-periodic orbit
$\gamma$ of $H$ to the one-periodic orbit
$\Phi_G(\gamma)(t):=\varphi^t_G(\gamma(t))$ of $G\# H$; see, e.g.,
\cite{Gi:conley} for more details.  Hence, we obtain an isomorphism of
Floer homology:
$$
\HF^{(a,\,b)}_*(H)\stackrel{\cong}\longrightarrow\HF^{(a+c,\,b+c)}_*(G\#H).
$$
As a consequence, the filtered Floer homology of $H$ is
determined by $\varphi_H$ up to a shift of the action filtration.

\subsection{The mean index}
\labell{sec:mean-index}

Let $\gamma$ be a one-periodic orbit of a Hamiltonian $H$ on $M$.  (It
suffices to have $H$ defined only on a neighborhood of $\gamma$.) The
mean index $\Delta_H(\gamma)\in\R$ measures the sum of rotations of
the eigenvalues of $d(\varphi^t_H)_{\gamma(t)}$ lying on the unit
circle. Here $d(\varphi^t_H)_{\gamma(t)}$ is interpreted as a path in
the group of linear symplectomorphisms by using, as above, the
trivialization of $TM$ along $\gamma$, associated with a disk bounded
by $\gamma$.  Referring the reader to \cite{SZ} for a precise
definition of $\Delta_H(\gamma)$ and the proofs of its properties, we
just recall here the following facts that are used in this paper.

\begin{itemize}

\item[(MI1)] The iteration formula:
$\Delta_{H^{\# k}}(\gamma^k)=k\Delta_H(\gamma)$.

\item[(MI2)] Continuity: Let $\tH$ be a $C^2$-small perturbation of
$H$ and let $\tgamma$ be a one-periodic orbit of $\tH$ close to
$\gamma$. Then $|\Delta_H(\gamma)-\Delta_{\tH}(\tgamma)|$ is small.

\item[(MI3)] The mean index formula: Assume that $\gamma$ is
non-degenerate. Then, as $k\to\infty$ through admissible iterations,
$\MUCZ(H^{\# k},\gamma^k)/k\to \Delta_H(\gamma)$.

\item[(MI4)] Relation to the Conley--Zehnder index: Let $\gamma$ split
into non-degenerate orbits $\gamma_1,\ldots,\gamma_m$ under a
$C^2$-small, non-degenerate perturbation $\tH$ of $H$. Then
$|\MUCZ(\tH,\gamma_i)-\Delta_H(\gamma)|\leq n$ for all $i=1,\ldots,m$.
Moreover, these inequalities are strict when $\gamma$ is weakly
non-degenerate; see \cite[p.\ 1357]{SZ}.  In particular, if $\gamma$
is non-degenerate, $|\MUCZ(H,\gamma)-\Delta_H(\gamma)|< n$.

\item[(MI5)] Additivity: Let $\gamma_1$ and $\gamma_2$ be one-periodic
orbits of Hamiltonians $H_1$ and $H_2$ on manifolds $M_1$ and,
respectively, $M_2$.  Then
$\Delta_{H_1+H_2}((\gamma_1,\gamma_2))=\Delta_{H_1}(\gamma_1)
+\Delta_{H_2}(\gamma_2)$, where $H_1+H_2$ is the naturally defined
Hamiltonian on $M_1\times M_2$.

\item[(MI6)] Action of global loops: Assume that $G$ generates a loop
of Hamiltonian diffeomorphisms of $M$. Then $\Delta_{G\#
H}(\Phi_G(\gamma))=\Delta_H(\gamma)$.

\item[(MI7)] Action of local loops: Assume that $\gamma(t)\equiv p$ is
a constant one-periodic orbit and that $G$ generates a loop of
Hamiltonian diffeomorphisms fixing $p$ and defined on a neighborhood
of $p$. Then $\Delta_{G\# H}(\Phi_G(\gamma))=\Delta_H(\gamma)+2\mu$,
where $\mu$ is the Maslov index of the loop $d(\varphi^t_G)_p$.

\item[(MI8)] Index of strongly degenerate orbits: Assume that $\gamma$
is strongly degenerate. Then $\Delta_H(\gamma)\in 2\Z$. Moreover, when
$\gamma\equiv p$ is a constant orbit and $H$ is defined on a
neighborhood of $p$ and generates a loop of Hamiltonian
diffeomorphisms, we have $\Delta_H(p)=2\mu$, where $\mu$ is the Maslov
index of the loop $d(\varphi^t_H)_p$.

\end{itemize}

\begin{Remark}
\labell{rmk:maslov}

Regarding (MI6) and (MI7), note that, as has been pointed out above,
the Maslov index of a global loop, in contrast with the index of a
local loop, is automatically zero. This ensures that the correction
term $2\mu$ vanishes in the setting of (MI6).
\end{Remark}

\section{Local Floer homology}
\labell{sec:LFH}

In this section, we briefly recall the definition and basic properties
of local Morse and Floer homology following mainly \cite{Gi:conley},
although these constructions go back to the original work of Floer (see,
e.g., \cite{F:witten,Fl}) and have been revisited a number of times
since then.

\subsection{Local Morse homology}
\labell{sec:LMH}

Let $f\colon M^m\to\R$ be a smooth function on a manifold $M$ and let
$p\in M$ be an isolated critical point of $f$. Fix a small
neighborhood $U$ of $p$ containing no other critical points of $f$ and
consider a small generic perturbation $\tf$ of $f$ in $U$. To be more
precise, $\tf$ is required to be Morse inside $U$ and $C^1$-close to
$f$. Then, as is easy to see, every anti-gradient trajectory
connecting two critical points of $\tf$ in $U$ is entirely contained
in $U$. Moreover, the same is true for broken trajectories. As a
consequence, the vector space (over $\Z_2$) generated by the critical
points of $\tf$ in $U$ is a complex with (Morse) differential defined
in the standard way; see, e.g., \cite{Jo,schwarz:book}. Furthermore,
the continuation argument shows that the homology of this complex,
denoted here by $\HM_*(f,p)$ and referred to as the \emph{local Morse
homology} of $f$ at $p$, is independent of the choice of $\tf$. This
construction is a particular case of the one from, e.g.,
\cite{F:witten}.

\begin{Example}
Assume that $p$ is a non-degenerate critical point of $f$ of index
$k$. Then $\HM_l(f,p)=\Z_2$ when $l=k$ and $\HM_l(f,p)=0$ otherwise.
\end{Example}

\begin{Example}
\labell{exam:Morse-max}
When $p$ is a strict local maximum of $f$, we have
$\HM_m(f,p)=\Z_2$. Indeed, in this case, as is easy to see from the
standard Morse theory,
$$
\HM_m(f,p)= H_m(\{f\geq f(p)-\eps\}, \{f= f(p)-\eps\})=\Z_2,
$$
where $\eps>0$ is assumed to be small and such that $f(p)-\eps$ is a
regular value of $f$. Furthermore, the converse is also true. In fact,
$f$ has (strict) local maximum at $p$ if and only if
$\HM_m(f,p)=\Z_2$; see, e.g., \cite{Gi:conley}.
\end{Example}

We will need the following property of local Morse homology, which can
be easily established by the standard continuation argument; cf.\
\cite{schwarz:book}.

\begin{itemize}
\item Let $f_s$, $s\in [0,\, 1]$, be a family of smooth functions with
\emph{uniformly isolated} critical point $p$, i.e., $p$ is the only
critical point of $f_s$, for all $s$, in some neighborhood of
$p$. Then $\HM_*(f_s,p)$ is constant throughout the family, i.e.,
$\HM_*(f_0,p)=\HM_*(f_1,p)$.
\end{itemize}

\begin{Remark}
  In this observation, the assumption that $p$ is uniformly isolated is
  essential and cannot be replaced by the weaker condition that $p$ is
  just an isolated critical point of $f_s$ for all $s$. Example:
  $f_s(x)=s x^2+ (1-s)x^3$ on $\R$ with $p=0$. (The authors are
  grateful to Doris Hein for this remark.)
\end{Remark}

\subsection{Local Floer homology: the definition and basic properties}
\labell{sec:LFH2}

Let $\gamma$ be an isolated one-periodic orbit of a Hamiltonian
$H\colon S^1\times M\to \R$. Pick a sufficiently small tubular
neighborhood $U$ of $\gamma$ and consider a non-degenerate $C^2$-small
perturbation $\tH$ of $H$ supported in $U$.  Every (anti-gradient)
Floer trajectory $u$ connecting two one-periodic orbits of $\tH$ lying
in $U$ is also contained in $U$, provided that $\|\tH-H\|_{C^2}$ and
$\supp(\tH-H)$ are small enough. (This readily follows from the
analysis carried out in, e.g., \cite{FHS,Sa:london,Sa}.)  Thus, by the
compactness and gluing theorems, every broken anti-gradient trajectory
connecting two such orbits also lies entirely in $U$. Hence, similarly
to the definition of local Morse homology, the vector space (over
$\Z_2$) generated by one-periodic orbits of $\tH$ in $U$ is a complex
with (Floer) differential defined in the standard way. The
continuation argument (see, e.g., \cite{SZ}) shows that the homology
of this complex is independent of the choice of $\tH$ and of the
almost complex structure. We refer to the resulting homology group
$\HF_*(H,\gamma)$ as the \emph{local Floer homology} of $H$ at
$\gamma$. The definition of local Floer homology and most of its
properties discussed below extend with natural modifications
to all symplectic manifolds, once the orbit $\gamma$ is equipped with
a capping; cf.\ \cite[Section 6.3.1]{GG}.

Homology groups of this type were first considered (in a more general
setting) by Floer in \cite{F:witten,Fl}; see also \cite[Section
3.3.4]{Poz}. Local Floer and Morse homology groups are analogues of
(non-equivariant) critical modules introduced in \cite{GM1,GM2}; see
also, e.g., \cite{Lo} for further references.

\begin{Example}
Assume that $\gamma$ is non-degenerate and $\MUCZ(\gamma)=k$.  Then
$\HF_l(H,\gamma)=\Z_2$ when $l=k$ and $\HF_l(H,\gamma)=0$ otherwise.
\end{Example}

In the rest of this section, we list the basic properties of local
Floer homology that are essential for what follows.

\begin{enumerate}
\item[(LF1)] Let $H^s$, $s\in [0,\, 1]$, be a family of Hamiltonians
  such that $\gamma$ is a \emph{uniformly isolated} one-periodic orbit
  for all $H^s$, i.e., $\gamma$ is the only periodic orbit of $H_s$,
  for all $s$, in some open set independent of $s$. Then
  $\HF_*(H^s,\gamma)$ is constant throughout the family:
  $\HF_*(H^0,\gamma)=\HF_*(H^1,\gamma)$.
\end{enumerate}

The proof of this fact is a straightforward application of the
continuation argument; see, e.g., \cite{Gi:conley}. As in the case of
local Morse homology, the condition that $\gamma$ is uniformly isolated
is essential.

Local Floer homology spaces are building blocks for filtered Floer
homology. Namely, essentially by definition, we have the following:

\begin{enumerate}
\item[(LF2)] Assume that $M$ is closed and let $c\in \R$ be such that
all one-periodic orbits $\gamma_i$ of $H$ with action $c$ are
isolated. (As a consequence, there are only finitely many orbits with
action close to $c$.) Then, if $\eps>0$ is small enough,
$$
\HF_*^{(c-\eps,\,c+\eps)}(H)=\bigoplus_i \HF_*(H,\gamma_i).
$$
In particular, if all one-periodic orbits $\gamma$ of $H$ are isolated
and $\HF_k(H,\gamma)=0$ for some $k$ and all $\gamma$, we have
$\HF_k(H)=0$ by the long exact sequence of filtered Floer homology.
\end{enumerate}

The local Floer homology is completely determined by the time-one map
generated by $H$:

\begin{enumerate}
\item[(LF3)] Let $\varphi^t_G$ be a loop of Hamiltonian
diffeomorphisms of $M$. Then
$$
\HF_*(G\# H,\Phi_G(\gamma)) = \HF_*(H,\gamma)
$$
for every isolated one-periodic orbit $\gamma$ of $H$.
\end{enumerate}

Hence, we will sometimes use the notation $\HF_*(\varphi,p)$ for
$\HF_*(H,\gamma)$, where $\varphi=\varphi_H^1$ and $p=\gamma(0)$; cf.\
Section \ref{sec:floer}.

Furthermore, the K\"unneth formula holds for local Floer homology:

\begin{enumerate}
\item[(LF4)] Let $\gamma_1$ and $\gamma_2$ be one-periodic orbits of
Hamiltonians $H_1$ and $H_2$ on, respectively, symplectic manifolds
$M_1$ and $M_2$. Then $\HF_*(H_1+H_2, (\gamma_1,\gamma_2))
=\HF_*(H_1,\gamma_1)\otimes \HF_*(H_2,\gamma_2)$, where $H_1+H_2$ is
the naturally defined Hamiltonian on $M_1\times M_2$.
\end{enumerate}
The proof of (LH4) is identical to the proof of the K\"unneth formula
for Floer homology.

By definition, the \emph{support} of $\HF_*(H,\gamma)$ is the
collection of integers $k$ such that $\HF_k(H,\gamma)\neq 0$. Clearly,
the group $\HF_*(H,\gamma)$ is finitely generated and hence supported
in a finite range of degrees. The next observation, providing more
precise information on the support of $\HF_*(H,\gamma)$, is an
immediate consequence of (MI4).

\begin{enumerate}
\item[(LF5)] The group $\HF_*(H,\gamma)$ is supported in the range
$[\Delta_H(\gamma)-n,\, \Delta_H(\gamma)+n]$. Moreover, when $\gamma$
is weakly non-degenerate, the support is contained in the open
interval $(\Delta_H(\gamma)-n,\, \Delta_H(\gamma)+n)$.
\end{enumerate}

As is clear from the definition of local Floer homology, $H$ need not
be a function on the entire manifold $M$ -- it is sufficient to
consider Hamiltonians defined only on a neighborhood of $\gamma$. For
the sake of simplicity, we focus on the particular case, relevant
here, where $\gamma(t)\equiv p$ is a constant orbit, and hence
$dH_t(p)=0$ for all $t\in S^1$. Then (LH1), (LH4) and (LF5) still
hold, and (LF3) takes the following form:

\begin{enumerate}
\item[(LF6)] Let $\varphi^t_G$ be a loop of Hamiltonian
diffeomorphisms defined on a neighborhood of $p$ and fixing $p$. Then
$$
\HF_*(G\# H,p) = \HF_{*+2\mu}(H,p),
$$
where $\mu$ is the Maslov index of the loop
$t\mapsto d(\varphi^t_G)_p\in \Sp(T_pM)$.
\end{enumerate}

Note that in (LF3), in contrast with (LH6), we \emph{a priori} have
$\mu=0$; cf.\ Remark \ref{rmk:maslov}. Hence, the shift of degree
does not occur when $\varphi^t_G$ is a global loop.  In other words,
comparing (LH3) and (LH6), we can say that while the group
$\HF_*(H,\gamma)$ is completely determined by $\varphi_H\colon M\to M$
and $p=\gamma(0)$, the germ of $\varphi_H$ at $p$ determines
$\HF_*(H,p)$ only up to a shift of degree. The degree depends on the
class of $\varphi_H^t$ in the universal covering of the group of germs
of Hamiltonian diffeomorphisms.

\subsection{Calculation of local Floer homology via local Morse homology}
\labell{sec:morse-floer}

A fundamental property of Floer homology is that $\HF_*(H)$ is equal
to the Morse homology of $H$ when $H$ is autonomous and $C^2$-small;
see \cite{FHS,SZ}. A similar identification holds for local Floer
homology.  In what follows, we will need a slightly more general
version of this fact, where the Hamiltonian is, in a certain sense,
``nearly'' autonomous.

\begin{Lemma}[\cite{Gi:conley}]
\labell{lemma:LFH-LMH}

Let $F$ be a smooth function and let $K$ be a one-periodic
Hamiltonian, both defined on a neighborhood of a point $p$. Assume
that $p$ is an isolated critical point of $F$ and the following
conditions are satisfied:
\begin{itemize}

\item The inequalities $\| X_{K_t}-X_F\|\leq\eps \| X_F\|$ and
$\|\dot{X}_{K_t} \|\leq\eps \| X_F\|$ hold point-wise near $p$ for all
$t\in S^1$. (The dot stands for the derivative with respect to time.)

\item The Hessians $d^2 (K_t)_p$ and $d^2 F_p$ and the constant
$\eps\geq 0$ are sufficiently small. Namely, $\eps<1$ and
$$
\eps(1-\eps)^{-1}+\max_{t}\|d^2 (K_t)_p\|
<2\pi\text{ and }  \eps(1-\eps)^{-1}+\|d^2 F_p\|< 2\pi.
$$
\end{itemize}
Then $p$ is an isolated one-periodic orbit of $K$ and
$\HF_*(K,p)=
\HM_{*+n}(F,p)$.
\end{Lemma}

\begin{Example}
\labell{ex:LFH-LMH}
Assume that $p$ is an isolated critical point of an autonomous
Hamiltonian $F$ and $ \|d^2 F_p\|<2\pi$. Then
$\HF_*(F,p)=\HM_{\ast+n}(F,p)$.
\end{Example}

To prove Lemma \ref{lemma:LFH-LMH}, one first shows that $p$ is a
uniformly isolated one-periodic orbit for all Hamiltonians from a
linear homotopy connecting $K$ and $F$. Thus,
$\HF_*(K,p)=\HF_{*}(F,p)$ by (LF1). Furthermore, $F$ is $C^2$-small
near $p$, and thus, by the standard argument (see, e.g.,
\cite{FHS,SZ}), $\HF_*(F,p)=\HM_{*+n}(F,p)$. We refer the reader to
\cite{Gi:conley} for a detailed argument.

\begin{Remark}
The requirement in Lemma \ref{lemma:LFH-LMH} that $K$ is
close to an autonomous Hamiltonian $F$ plays a crucial role in
the proof of the Conley conjecture, \cite{Gi:conley,Hi}. To the best
of the authors' knowledge, this requirement is originally introduced
in \cite[Lemma 4]{Hi}.
\end{Remark}

\section{Persistence of local Floer homology}
\labell{sec:persist}

The main objective of this section is to prove Theorem
\ref{thm:persist-lf}. Since the question is local, we may assume
without loss of generality that $M=\R^{2n}$ and $\gamma\equiv p=0$ is
a constant one-periodic orbit of a germ of a Hamiltonian $H$.  Indeed,
it is easy to show that the path $\varphi^t_H$, $t\in [0,\,1]$, is
homotopic with fixed end-points to a path $\varphi_{\tH}^t$ such that
$\varphi^t_{\tH}(p)=p$ for all $t$; see \cite[Sections 2.3 and
5.1]{Gi:conley}. (The argument goes through for a general, not
necessarily symplectically aspherical, manifold.) Then $H$ and $\tH$
have isomorphic graded local Floer homology groups at $p$, and we can
just restrict $\tH$ to a neighborhood of $p$ and use the Darboux
theorem. Note also that $p$ is a critical point of ${\tH}_t$ for all
$t$.  From now on, we revert to the notation $H$ for the Hamiltonian
generating the flow near $p$ and set $\varphi=\varphi_H$. The fact
that $\gamma^k\equiv p$ is isolated follows from
Proposition~\ref{prop:persist-iso}.

The proof of Theorem \ref{thm:persist-lf} rests on two building
blocks. These are the (nearly obvious) case where the fixed point is
non-degenerate and the much less trivial case of a strongly degenerate
fixed point. Then the K\"unneth formula implies that the theorem also
holds for a split map, i.e., a Hamiltonian diffeomorphism that can be
decomposed as the direct product of non-degenerate and strongly
degenerate ones.  Finally, the general case is established by showing
that $\varphi$ can be deformed to a split map in the class of
Hamiltonian diffeomorphisms with isolated fixed point at the origin.
The ``moreover'' part of the theorem asserting that $p$ is strongly
degenerate and all shifts $s_k$ are equal to zero if $\Delta_H(p)=0$
and $\HF_n(H,p)\neq 0$ is proved in Section \ref{sec:pf-lf-degen}.

Also note that the fact that $s_k/k\to \Delta_H(\gamma)$ is clear once
\eqref{eq:case2} has been established.  Indeed, pick $l$ such that
$\HF_l(H,\gamma)\neq 0$. Then $\HF_{l_k}(H^{\# k},\gamma^k)\neq 0$,
where $l_k=l+s_k$ by \eqref{eq:case2}.  Furthermore,
$|l_k-\Delta_{H^{\# k}}(\gamma^k)|\leq n$ and $\Delta_{H^{\#
k}}(\gamma^k)=k\Delta_H(\gamma)$ by (LF5) and the iteration formula
(MI1). To summarize, $|s_k+l-k\Delta_H(\gamma)|\leq n$. Dividing this
inequality by $k$, we see that $s_k/k\to \Delta_H(\gamma)$. Moreover,
$|s_k/k-\Delta_H(\gamma)|\leq (n+l)/k$, where
$|l-\Delta_H(\gamma)|\leq n$.

\subsection{Particular case 1: $p$ is non-degenerate}
In this case, the assertion is obvious. Namely, $p$ is a
non-degenerate fixed point of $\varphi^k$ for every admissible $k$,
and hence the Conley-Zehnder index $\mu_k$ of $\varphi^k$ at $p$ is
defined. Clearly,
$$
\HF_l(\varphi^k,p)=\left\{
\begin{aligned}
\Z_2 &\quad\text{if } l=\mu_k,\\
0 &\quad\text{otherwise,}
\end{aligned}
\right.
$$
and the shifts $s_k=\mu_k-\mu_1$ are even when $k$ is good; see \cite{SZ}.

\subsection{Generating functions}
Before turning to the next particular case -- that of a strongly
degenerate fixed point -- we recall in this section a few
well-known facts concerning generating functions, which are utilized
in Section \ref{sec:pf-lf-degen}. The material reviewed here is
absolutely standard -- it goes back to Poincar\'e -- and we refer the
reader to \cite[Appendix 9]{Ar} and \cite{We71,We77} for more details.

Let us identify $\R^{2n}$ with the Lagrangian diagonal $\Delta\subset
\R^{2n}\times \bar{\R}^{2n}$ via the projection $\pi$ to the first
factor, where $\R^{2n}\times \bar{\R}^{2n}$ is equipped with the
symplectic structure $\omega\oplus -\omega$, and fix a Lagrangian
complement $N$ to $\Delta$.  Thus, $\R^{2n}\times \bar{\R}^{2n}$ can
now be treated as~$T^*\Delta$.

Let $\varphi$ be a Hamiltonian diffeomorphism defined on a
neighborhood of the origin $p$ in $\R^{2n}$ and such that
$\|\varphi-\id\|_{C^1}$ is sufficiently small. Then the graph $\Gamma$
of $\varphi$ is $C^1$-close to $\Delta$, and hence $\Gamma$ can be
viewed as the graph in $T^*\Delta$ of an exact form $dF$ near
$p\in\Delta=\R^{2n}$.  The function $F$, normalized by $F(p)=0$ and
called the \emph{generating function} of $\varphi$, has the following
properties:
\begin{enumerate}
\item[(GF1)] $p$ is an isolated critical point of $F$ if and only if
$p$ is an isolated fixed point of $\varphi$,

\item[(GF2)] $\|F\|_{C^2}=O(\|\varphi-\id\|_{C^1})$ and
$\| d^2 F_p\|=O(\|d\varphi_p-I\|)$.
\end{enumerate}

The function $F$ depends on the choice of the Lagrangian complement
$N$ to $\Delta$. To be specific, we take, as $N$, the linear subspace
$N_0$ of vectors of the form $((x,0),(0,y))$ in $\R^{2n}\times
\bar{\R}^{2n}$, where $x=(x_1,\ldots, x_n)$ and $y=(y_1,\ldots,y_n)$
are the standard canonical coordinates on $\R^{2n}$, i.e.,
$\omega=\sum dy_i\wedge dx_i$.  Set $z=(x,y)$ and
$$
\psi(z):=(\text{$x$-component of $\varphi(z)$},y).
$$
Then, as is easy to see, $F$ is determined by the equation
$$
\varphi(z)-z=X_F(\psi(z)),
$$
where $X_F$ is the Hamiltonian vector field of $F$. Note also that
$N_0$, and hence $F$, are uniquely determined by the decomposition of
$\R^{2n}$ into the direct sum of two Lagrangian subspaces -- the
subspace spanned by $x$-coordinates and the one spanned by
$y$-coordinates. Therefore, fixing two transverse Lagrangian subspaces
in $\R^{2n}$ gives rise to a generating function of $\varphi$.

The only reason that above we assumed $\varphi$ to be $C^1$-close to
$\id$ is to make $N$ independent of $\varphi$, and hence make the
construction of $F$ to some extent canonical. This assumption can be
dropped once more flexibility in the choice of $N$ is allowed.
Namely, as is easy to see, for any germ $\varphi$ there exists a
Lagrangian complement $N$ to $\Delta$, transverse to the graph of
$\varphi$.  Then $\varphi$ is given by a generating function with
respect to $N$. Conversely, once $N$ is fixed, every function $F$ is
the generating function of some Hamiltonian diffeomorphism, provided
that the graph of $dF$ in $\R^{2n}\times \bar{\R}^{2n}=\Delta\times
N=T^*\Delta$ is transverse to the fibers of the projection $\pi\colon
\R^{2n}\times \bar{\R}^{2n}\to \R^{2n}$.

\begin{Remark}
\labell{rmk:loc-ham}
From these observations, we recover the well-known fact, used in
Section \ref{sec:split}, that the germ of any symplectomorphism
$\varphi$ near a fixed point is Hamiltonian. Indeed, let $F$ be the
generating function of $\varphi$ with respect to some $N$. Set
$F_s=(1-s)F+ sd^2F_p$ with $s\in [0,\,1]$.  Since $d^2 (F_s)_p=d^2
F_p$, the graph of $dF_s$ is transverse to the fibers of $\pi$ for all
$s$, and we obtain a family $\varphi_s$ of symplectic maps fixing $p$
and connecting $\varphi$ to the linear symplectic map $d\varphi$.  As
a consequence, the germ $\varphi$ lies in the connected component of
the identity, and thus, by the standard, elementary argument,
$\varphi$ is Hamiltonian.
\end{Remark}

\subsection{Particular case 2: $p$ is strongly degenerate}
\labell{sec:pf-lf-degen}

Since, by definition, all eigenvalues of $d\varphi_p$ are equal to
one, every $k>0$ is admissible and good.  Furthermore, as is easy to check, by
a linear symplectic change of coordinates one can make $d\varphi_p$
arbitrarily close to identity; see \cite[Section 5.2.1]{Gi:conley}.
Hence, we may assume without loss of generality that the iterations
$\varphi^k$ for all $k$ in an arbitrarily large, but fixed, range
are $C^1$-close to $\id$ in a sufficiently small neighborhood of
$p$. As a consequence, $\varphi^k$ is given by a generating function
$F_k$ with respect to $N_0$, which is uniquely determined by the
equation
$$
\varphi^k(z)-z=X_{F_k}(\psi_k(z)),\quad F_k(p)=0,
$$
where, as above,
$$
\psi_k(z):=(\text{$x$-component of $\varphi^k(z)$},y).
$$
Set $F=F_1$ and $G_t=tF_k+(1-t)kF$, where $t\in [0,\,1]$.

\begin{Claim}
The origin $p$ is a uniformly isolated critical point of the family $G_t$,
$t\in [0,\,1]$.
\end{Claim}

Assuming the claim, let us proceed with the proof. First recall that,
starting with $F_k$, one can construct near $p$ a one-periodic
Hamiltonian $K^t_k$ with time-one map $\varphi^k$, satisfying the
hypotheses of Lemma \ref{lemma:LFH-LMH}; see \cite{Gi:conley,Hi}. Then
the local Floer homology of $K_k$ at $p$ is equal to the local Morse
homology of $F_k$ at $p$ up to a shift of degree by $n$:
$$
\HF_*(K_k,p)=\HM_{*+n}(F_k,p).
$$
The Hamiltonians $H^{\# k}$ and $K_k$ generate the same time-one map near
$p$. Thus,
$$
\HF_*(K_k,p)=\HF_{*+m_k}(H^{\# k},p),
$$
by (LF6), for some even shift of degree $m_k$.  From the claim and
homotopy invariance of local Morse homology (see (LF1)), we infer that
$$
\HM_{*}(F_k,p)=\HM_{*}(kF,p)=\HM_{*}(F,p).
$$
Hence,
$$
\HF_{*+m_k}(H^{\# k},p)=\HM_{*+n}(F_k,p)=\HM_{*+n}(F,p)=\HF_{*+m_1}(H,p),
$$
and thus \eqref{eq:case2} holds with $s_k=m_1-m_k$.
Since all $m_k$ are even, the shifts $s_k$ are also even.

Now we are in a position to prove the ``moreover'' part of the
theorem.  The fact that $p$ is strongly degenerate if $\Delta_H(p)=0$
and $\HF_n(H,p)\neq 0$ follows immediately from (MI4) or (LF5).  It
remains to show that $s_k=0$ for all $k$. By (MI7),
$$
m_k=\Delta_{K_k}(p)-\Delta_{H^{\# k}}(p).
$$
Here $\Delta_{H^{\# k}}(p)=k\Delta_H(p)=0$ and $|\Delta_{K_k}(p)|$ is
small since $d(\varphi^t_{K_k})_p$ is close to the identity. (To be
more precise, for a fixed $k$, the path of linear maps
$d(\varphi^t_{K_k})_p$ can be made arbitrarily close to the identity
by a symplectic conjugation.)  We conclude that $m_k=0$, since $m_k$
is an integer, and hence $s_k=0$.

\begin{proof}[Proof of the claim]
First, let us show that
\begin{equation}
\labell{eq:Xk}
\| X_{F_k}(\psi_k(z))-kX_{F}(\psi(z))\|=O(\|\varphi-\id\|_{C^1})
\|X_{F}(\psi(z))\|,
\end{equation}
where $\psi=\psi_1$ and $\|\cdot\|_{C^1}$ stands for the $C^1$-norm on
a sufficiently small ball centered at the origin. Then, as a
consequence of \eqref{eq:Xk}, we have
\begin{equation}
\labell{eq:Xk2}
\| X_{F_k}(\psi_k(z))\|=\big(k+O(\|\varphi-\id\|_{C^1})\big)
\|X_{F}(\psi(z))\|.
\end{equation}

To prove \eqref{eq:Xk}, we argue inductively. When $k=1$, the left
hand side is zero and the assertion is obvious. Assume that
\eqref{eq:Xk}, and hence \eqref{eq:Xk2}, have been established for all
iterations of order up to and including $k-1$. Then
\begin{equation*}
\begin{split}
X_{F_k}(\psi_k(z)) &= \varphi^k(z)-z\\
&= \big(\varphi^k(z)-\varphi^{k-1}(z)\big)+\ldots+\big(\varphi(z)-z\big)\\
&= X_{F}(\psi\varphi^{k-1}(z))+\ldots+ X_{F}(\psi(z)),
\end{split}
\end{equation*}
and therefore
\begin{equation}
\labell{eq:sum}
\begin{split}
X_{F_k}(\psi_k(z))-kX_{F}(\psi(z))
& =
X_{F}(\psi\varphi^{k-1}(z))-X_{F}(\psi(z))\\
&\quad +\ldots\\
&\quad + X_{F}(\psi\varphi(z))-X_{F}(\psi(z)).
\end{split}
\end{equation}
Furthermore, for every $l$ in the range from 1 to $k-1$, we have
\begin{equation*}
\begin{split}
\|X_{F}(\psi\varphi^{l}(z))-X_{F}(\psi(z))\|
&\leq \|X_F\|_{C^1}\cdot\|\psi\varphi^{l}(z)-\psi(z)\|\\
&\leq \|X_F\|_{C^1}\cdot\|\psi\|_{C^1}\cdot\|\varphi^{l}(z)-z\|\\
&= \|X_F\|_{C^1}\cdot\|\psi\|_{C^1}\cdot\|X_{F_l}(\psi_l(z))\|.
\end{split}
\end{equation*}
It is clear that
$\|X_F\|_{C^1}=O(\|F\|_{C^2})=O(\|\varphi-\id\|_{C^1})$ by (GF2) and
$\|\psi\|_{C^1}\leq\const$. Finally, by the induction hypothesis,
$$
\|X_{F_l}(\psi_l(z))\|\leq \big(l+O(\|\varphi-\id\|_{C^1})\big)
\|X_{F}(\psi(z))\|.
$$
As a consequence,
$$
\|X_{F}(\psi\varphi^{l}(z))-X_{F}(\psi(z))\|\leq
O(\|\varphi-\id\|_{C^1})\|X_{F}(\psi(z))\|
$$
for all $l=1,\ldots,k-1$. Adding up these upper bounds for
$l=1,\ldots,k-1$ and using \eqref{eq:sum}, we obtain \eqref{eq:Xk}.

Continuing the proof of the claim, we note that it is sufficient to
show that $p$ is a uniformly isolated zero of $X_{G_t}$. Clearly, for any
vector field $Y_t$,
\begin{equation}
\labell{eq:XY}
\|X_{G_t}(z)\|\geq \|Y_t(z)\|-\|Y_t(z)-X_{G_t}(z)\|.
\end{equation}
Using the linear structure on $\R^{2n}$, we set
$$
Y_t(z)=tX_{F_k}(\psi_k(z))+(1-t)k\cdot X_F(\psi(z))
$$
and bound the first term on the right hand side of \eqref{eq:XY} from
below and the second term from above.

By \eqref{eq:Xk} and the definition of $Y_t(z)$,
\begin{equation}
\labell{eq:Y}
\begin{split}
\|Y_t(z)\| &\geq k\|X_F(\psi(z))\|-\|X_{F_k}(\psi_k(z))-kX_F(\psi(z))\|\\
&\geq \big(k-O(\|\varphi-\id\|_{C^1})\big)\|X_F(\psi(z))\|.
\end{split}
\end{equation}
Next we show that
\begin{equation}
\labell{eq:YX}
\|Y_t(z)-X_{G_t}(z)\|\leq O(\|\varphi-\id\|_{C^1})\|X_F(\psi(z))\|.
\end{equation}
To this end, first note that
\begin{equation}
\labell{eq:zpsi}
\begin{split}
\|X_F(z)-X_F(\psi(z))\| &\leq \|X_F\|_{C^1}\cdot\|\psi(z)-z\|\\
&\leq \|X_F\|_{C^1}\cdot\|\varphi(z)-z\|\\
&\leq \|X_F\|_{C^1}\cdot\|X_F(\psi(z))\|\\
&= O(\|\varphi-\id\|_{C^1})\|X_F(\psi(z))\|,
\end{split}
\end{equation}
where the second inequality readily follows from the definition of
$\psi$. Similarly,
$$
\|X_{F_k}(z)-X_{F_k}(\psi_k(z))\|
\leq O(\|\varphi-\id\|_{C^1})\|X_{F_k}(\psi_k(z))\|.
$$
Hence, by \eqref{eq:Xk2},
\begin{equation}
\labell{eq:zpsik}
\|X_{F_k}(z)-X_{F_k}(\psi_k(z))\|
\leq O(\|\varphi-\id\|_{C^1})\|X_{F}(\psi(z))\|.
\end{equation}
Furthermore,
\begin{equation*}
\begin{split}
\|Y_t(z)-X_{G_t}(z)\|
&\leq t\|X_{F_k}(z)-X_{F_k}(\psi_k(z))\| +(1-t)k\|X_F(z)-X_F(\psi(z))\|\\
&\leq \|X_{F_k}(z)-X_{F_k}(\psi_k(z))\| +k\|X_F(z)-X_F(\psi(z))\|.
\end{split}
\end{equation*}
Combining this with \eqref{eq:zpsi} and \eqref{eq:zpsik}, we obtain
\eqref{eq:YX}. Finally, using the bounds \eqref{eq:Y} and
\eqref{eq:YX} and the inequality \eqref{eq:XY}, we conclude that
$$
\|X_{G_t}(z)\|\geq \big(k-O(\|\varphi-\id\|_{C^1})\big)\|X_F(\psi(z))\|.
$$
It is immediate to see that $\psi$ is a diffeomorphism on a
sufficiently small neighborhood of the origin and $\psi(p)=p$. Hence,
$\psi(z)=p$ implies that $z=p$. Furthermore, $p$ is a uniformly isolated zero
of $X_F$ by (GF1). Thus, $p$ is also a uniformly isolated zero of
$X_{G_t}$. This completes the proof of the claim.
\end{proof}

\subsection{Particular case 3: split diffeomorphisms}
\labell{sec:split}

Assume that $\R^{2n}$ is decomposed as a product of two symplectic
vector spaces $V$ and $W$ and $H$ is also split, i.e., $H=H_V+H_W$,
where $H_V$ and $H_W$ are Hamiltonians on $V$ and, respectively, $W$
with flows fixing the origin. Assume, in addition, that the time
one-map $\varphi_{H_V}$ of $H_V$ is non-degenerate and the time-one
map $\varphi_{H_W}$ of $H_W$ is strongly degenerate. Then combining
the previous two particular cases and applying the K\"unneth formula
for local Floer homology (see (LF4)), we conclude that the theorem
holds for~$H$.

More generally, assume that $\varphi$, but not necessarily $H$, is
split, i.e., $\varphi_H=(\varphi_{V},\varphi_{W})$. Then $\varphi_{V}$
and $\varphi_{W}$ are germs of symplectomorphisms fixing $p$, and
hence both $\varphi_{V}$ and $\varphi_{W}$ are Hamiltonian; see Remark
\ref{rmk:loc-ham}. As above, denote by $H_V$ and $H_W$ some
Hamiltonians generating $\varphi_{V}$ and, respectively,
$\varphi_{W}$. We do not necessarily have $H=H_V+H_W$, but since local
Floer homology is determined by $\varphi$ up to a shift of indices,
$$
\HF_*(H^{\# k},p)=\HF_{*+s'_k}(H^{\# k}_V+H^{\# k}_W,p).
$$

The Hamiltonian $H_V+H_W$ is split and, as has been shown above, the
theorem holds for $H_V+H_W$. It remains to prove that the additional
shifts $s'_k$ are even. This, however, is clear, for $s'_k/2$ is the
Maslov index of the loop obtained by taking the concatenation of the
flow of $H^{\# k}$ and the inverse flow of $H^{\# k}_V+H^{\# k}_W$.

\subsection{The general case}

Let $\varphi$ be the germ of the Hamiltonian diffeomorphism fixing the
origin $p$ in $\R^{2n}$ and generated by $H$. For some decomposition
$\R^{2n}=V\times W$ the linearization $d\varphi_p$ splits as the
direct sum of a symplectic linear map on $V$ whose eigenvalues are all
different from one and a symplectic linear map on $W$ with all
eigenvalues equal to one. Then, if $k$ is admissible, the same
splitting holds for $d\varphi^k_p$.  We will show that $\varphi$ is
homotopic to a split map via Hamiltonian diffeomorphisms
with uniformly isolated fixed point at $p$ and linearization
$d\varphi_p$. Denote such a homotopy by $\varphi_s$, $s\in
[0,\,1]$. Then $p$ is also a uniformly isolated fixed point for all
maps in the iterated homotopy $\varphi_s^k$ (see Remark
\ref{rmk:persist-iso} and Proposition \ref{prop:persist-iso2}) and the
theorem follows from Case 3 and the invariance of local Floer homology
under homotopy; see (LF1).

To be more precise, let $K_{s}$ be the Hamiltonian generating
$\varphi_s$ as its time-one map and obtained, up to an obvious
reparametrization, by concatenating the flow $\varphi_H^t$, $t\in
[0,\, 1]$, with the homotopy $\varphi_\zeta$, $\zeta\in [0,\,s]$. Then
$p$ is a uniformly isolated fixed point of $\varphi^k_{K_s}$ for all $s\in
[0,\,1]$ and all admissible $k$. Hence, by (LF1),
$$
\HF_*(H^{\# k},p)=\HF_*(K^{\# k}_1,p).
$$
In addition, $\Delta_H(p)=\Delta_{K_1}(p)$, for $d\varphi_s$ at $p$ is
constant.  Since $\varphi_1=\varphi_{K_1}$ is split, the theorem holds
for $K_1$.  Therefore, the theorem also holds for $H$.

Let us now construct the homotopy $\varphi_s$. Let $N_V$ and $N_W$ be
Lagrangian complements to the diagonals $\Delta_V\subset V\times
\bar{V}$ and, respectively, $\Delta_W\subset W\times \bar{W}$,
transverse to the graphs of $d\varphi_p\mid_V$ and $d\varphi_p\mid_W$.  Then
$N=N_V\times N_W$ is a Lagrangian complement to the diagonal $\Delta$
in $\R^{2n}\times \bar{\R}^{2n}$, transverse to the graph of
$d\varphi_p$, and hence to the graph of $\varphi$ on a small
neighborhood of $p$. Denote by $F$ the generating function of $\varphi$
with respect to $N$ on a neighborhood of $p$. Note that $p$ is an
isolated critical point of $F$ and $d^2 F_p$ is split.

We will construct a family of functions $F_s$, $s\in [0,\,1]$, on a
neighborhood of $p$ starting with $F_0=F$ and such that
\begin{itemize}

\item $p$ is a uniformly isolated critical point of $F_s$,

\item $d^2 (F_s)_p=d^2 F_p$,

\item $F_1$ is split, i.e., $F_1$ is the sum of a function $q$ on $V$
and a function $f$ on $W$ near $p$.
\end{itemize}

Once the family $F_s$ is constructed, $\varphi_s$ is defined in an
obvious way via identifying the graph of $\varphi_s$ with the graph of
$d F_s$ in $\R^{2n}\times \bar{\R}^{2n}=\Delta\times
N=T^*\Delta$. (The graph of $d F_s$ is transverse to the fibers of the
projection $\pi\colon \R^{2n}\times \bar{\R}^{2n}\to \R^{2n}$ near
$p$, since $d^2 (F_s)_p=d^2 F_p$ and the graph of $dF$, coinciding
with the graph of $\varphi$, is transverse to the fibers.) Note also
that in the decomposition $F_1=q+f$, the function $q$ is a
non-degenerate quadratic form on $V$ (in fact, $q=d^2 F_p\mid_V$) and $f$
is a function on $W$ with isolated critical point at the origin.

To find the family $F_s$, we argue as follows. First observe that, by
the implicit function theorem, there exists (near $p$) a unique smooth
map $\Phi\colon W\to V$ such that $\Phi(0)=0$ and $F\mid_{V\times w}$ has
a critical point at $\Phi(w)$. Let $\Sigma$ be the graph of $\Phi$. It
is easy to see that $d\Phi$ vanishes at the origin, for $d^2F_p$ is
split, and hence $\Sigma$ is tangent to $W$ at $p$. Now $F_s$ is
constructed in two steps. First, we use an isotopy on a neighborhood
of $p$, fixing $p$ and having the identity linearization at $p$, to
move $\Sigma$ to $W$. This isotopy turns $F$ into a function, say
$F_{0.5}$, such that $F_{0.5}\mid_{V\times w}$ has a non-degenerate
critical point at $(0,w)$ for all $w$ near the origin. As the second
step, we apply the parametric Morse lemma to $F_{0.5}\mid_{V\times w}$ to
obtain a homotopy from $F_{0.5}$ to a function $F_1$ of the desired
form $q+f$.

This concludes the proof of the theorem.

\section{Symplectically degenerate maxima}
\labell{sec:sdm-product}

Strongly degenerate periodic orbits with persistent Floer homology in
degree $n$, referred to in \cite{Gi:conley} as symplectically
degenerate maxima, play a particularly interesting role in the proof
of the Conley conjecture; see \cite{Gi:conley}. This role is further
exemplified by Theorem \ref{thm:persist-lf} and we feel that features
of such orbits merit further investigation.  In this section, we
characterize symplectically degenerate maxima in homological and
geometrical terms and then, in Section \ref{sec:product}, touch upon
``vanishing properties'' of the pair-of-pants product in local Floer
homology. Namely, we show that a periodic orbit is a symplectically
degenerate maximum if and only if the product is \emph{not} in a
certain sense nilpotent. The latter topic is rather tangential to the
main subject of the paper and is treated here very briefly, skipping
some technical details.

\subsection{Homological and geometrical properties of symplectically
degenerate maxima}
\labell{sec:sdm}

Let $\gamma$ be a one-periodic orbit of the flow of a Hamiltonian $H$
on a symplectically aspherical manifold $M^{2n}$. In fact, it suffices
to assume that $H$ is the germ of a Hamiltonian on a neighborhood of
$\gamma$.

\begin{Definition}
\labell{def:sdm}
The orbit $\gamma$ is said to be a \emph{symplectically degenerate
maximum} of $H$ if $\Delta_H(\gamma)=0$ and $\HF_n(H,\gamma)\neq 0$.
\end{Definition}

\begin{Example}
\labell{exam:sdm} Let $H$ be an autonomous Hamiltonian attaining a
strict local maximum at $p$.  Assume in addition that $d^2 H_p=0$ or,
more generally, that all eigenvalues of $d^2 H_p$ are zero.  Then it
is easy to see that $p$ is a symplectically degenerate maximum of $H$,
cf.\ Proposition \ref{prop:sdm2}. (Here, as is customary in
Hamiltonian dynamics, the eigenvalues of a quadratic form on a
symplectic vector space are, by definition, the eigenvalues of the
linear symplectic vector field generated by the quadratic form.)
\end{Example}

\begin{Proposition}
\labell{prop:sdm1}

The following conditions are equivalent:

\begin{itemize}

\item[(a)] the orbit $\gamma$ is a symplectically degenerate maximum of $H$;

\item[(b)] $\HF_n(H^{\# k_i},\gamma^{k_i})\neq 0$ for some sequence of
admissible iterations $k_i\to\infty$;

\item[(c)] the orbit $\gamma$ is strongly degenerate,
$\HF_n(H,\gamma)\neq 0$ and $\HF_n(H^{\# k},\gamma^k)\neq 0$ for at
least one admissible iteration $k\geq n+1$.

\end{itemize}

\end{Proposition}

\begin{proof}
The facts that (a) and (b) are equivalent and that (a) implies (c)
follow immediately from Theorem \ref{thm:persist-lf}. To show that (c)
implies (a), it is sufficient to prove that
$\Delta_H(\gamma)=0$. Assume the contrary.  Then
$|\Delta_H(\gamma)|\geq 2$ since $\Delta_H(\gamma)\in 2\Z$ due to the
assumption that $\gamma$ is strongly degenerate and (MI8). Thus,
$$
|\Delta_{H^{\# k}}(\gamma^k)|=k|\Delta_H(\gamma)|\geq 2k\geq 2(n+1).
$$
Therefore, by (LF5), $\HF_*(H^{\# k},\gamma^k)$ is supported in the
interval $[n+1,\, 3n+1]$, which contradicts the condition that
$\HF_n(H^{\# k},\gamma^k)\neq 0$.
\end{proof}

As a consequence of the proposition, we observe that, for any
admissible iteration $k$, the orbit $\gamma^k$ is a symplectically
degenerate maximum if and only if $\gamma$ is a symplectically
degenerate maximum.

To illuminate the geometrical nature of symplectically degenerate
maxima, let us assume that the orbit $\gamma$ is constant, i.e.,
$\gamma(t)\equiv p$ and $H$ is defined on a neighborhood of $p$. Then,
as our next result shows, the behavior of $\varphi_H$ near $p$ is
similar to that described in Example \ref{exam:sdm}. The essence of
this result is that $p$ is a symplectically degenerate maximum of $H$
if and only if $\varphi_H$ can be generated by a Hamiltonian $K$ with
local maximum at $p$ and arbitrarily small Hessian.

\begin{Proposition}
\labell{prop:sdm2}

The point $p$ is a symplectically degenerate maximum of $H$ if and
only if for every $\eps>0$ there exists a Hamiltonian $K$ near $p$
such that $\varphi_K=\varphi_H$ in the universal covering of the group
of local symplectomorphisms fixing $p$ and
\begin{itemize}
\item[(i)] $p$ is a strict local maximum of $K_t$ for all $t\in S^1$,
\item[(ii)] $\|d^2 (K_t)_p\|_{\Xi}<\eps$ for all $t\in S^1$ and some
symplectic basis $\Xi$ in $T_p M$.
\end{itemize}

\end{Proposition}

To clarify the terminology used here, recall that $\|d^2
(K_t)_p\|_{\Xi}$ stands for the norm of $d^2 (K_t)_p$ with respect to
the Euclidean inner product on $T_p M$ for which $\Xi$ is an
orthonormal basis; see \cite[Section 2.1.3]{Gi:conley}.

\begin{proof}[Proof of Proposition \ref{prop:sdm2}]
The non-trivial part of the proposition is that a Hamiltonian $K$ with
the required properties exists whenever $p$ is a symplectically
degenerate maximum. This is established in \cite[Proposition
4.5]{Gi:conley}.  Conversely, $\Delta_H(p)=\Delta_K(p)$ by (MI8). We
infer from (ii) that $|\Delta_K(p)|$ can be made arbitrarily small for
a suitable choice of $K$. Thus, $\Delta_H(p)=0$. Furthermore, using
(i), it is straightforward to construct a $C^2$-small perturbation
$\tilde{K}$ of $K$ such that $p$ is a non-degenerate local maximum of
$\tilde{K}_t$, and $\tilde{K}$ has no other one-periodic orbits near
$p$. Since $\|d^2 (\tilde{K}_t)_p\|_{\Xi}$ is small,
$\MUCZ(\tilde{K},p)=n$. Hence, $\HF_n(K,p)=\HF_n(\tilde{K},p)=\Z_2$.
\end{proof}

\begin{Remark}
It is clear from Propositions \ref{prop:sdm1} and \ref{prop:sdm2} that
the definition of a symplectically degenerate maximum given here is
equivalent to the one in \cite{Gi:conley}. (As a consequence, the
additional requirement (K3) in \cite[Definition 4.1]{Gi:conley} is
superfluous and follows from (K1) and (K2), reformulations of (i) and
(ii).) The proof of Proposition \ref{prop:sdm2} also shows that in
Definition \ref{def:sdm} and in (b) and (c) the conditions
$\HF_n(H,\gamma)\neq 0$ and $\HF_n(H^{\# k},\gamma^k)\neq 0$ can be
replaced by the more specific requirement that these Floer homology
groups are isomorphic to $\Z_2$.
\end{Remark}

The definition a symplectically degenerate maximum and Propositions
\ref{prop:sdm1} and \ref{prop:sdm2} extend word-for-word to isolated
fixed points of Hamiltonian diffeomorphisms $\varphi\colon M\to M$,
for the local Floer homology and the mean index are completely
determined by $\varphi$. When $\varphi$ is just the germ of a
Hamiltonian diffeomorphism near an isolated fixed point $p$, the
grading of local Floer homology and the mean index are defined only up
to a shift by the same even integer.  In this case, we say that $p$ is
a \emph{local} symplectically degenerate maximum when $\varphi$ can be
generated by a Hamiltonian $H$ with flow fixing $p$ and symplectically
degenerate maximum at $p$. By (MI8), (LF5) and (LF6), this is
equivalent to that $\Delta_K(p)\in \Z$ and
$\HF_{n+\Delta_K(p)}(K,p)\neq 0$ for any (or, equivalently, some)
Hamiltonian $K$ with $\varphi_K=\varphi$ and $\varphi_K^t(p)\equiv
p$. Furthermore, then $\Delta_K(p)$ is necessarily even. (Warning: a
fixed point $p$ of $\varphi\colon M\to M$ can be a local
symplectically degenerate maximum of the germ of $\varphi$ at $p$, but
not a symplectically degenerate maximum of~$\varphi$.)

\subsection{Product in local Floer homology}
\labell{sec:product}

The construction of the pair-of-pants product in Floer homology (see,
e.g., \cite{mdsa,PSS,Sc:thesis}) carries over in an obvious way to
local Floer homology. Thus, we have a product
$$
\underbrace{\HF_*(H,\gamma)\otimes \ldots \otimes\HF_*(H,\gamma)}_{r}
\to \HF_*(H^{\# r},\gamma^r),
$$
where $\gamma$ is an isolated one-periodic orbit of $H$ and $r$ is
admissible. When $u_i\in \HF_{l_i}(H,\gamma)$, the product
$u_1\cdot\ldots\cdot u_r$ has degree $l_1+\ldots+l_r-(r-1)n$, where
$\dim M=2n$. Up to a shit of degree, the product is a feature of the
germ of $\varphi=\varphi_H$ at the fixed point $\gamma(0)=p$.  Indeed,
assume for the sake of simplicity that $\gamma$ is constant.  Then
$\HF_*(H^{\# r},p)$ is isomorphic to $\HF_{*+m_r}(K^{\# r},p)$ for any
two Hamiltonians $H$ and $K$ generating $\varphi$ near $p$ and some
$m_r$. The isomorphism is induced by the composition with the
corresponding loop of Hamiltonian diffeomorphisms near $p$ and, as is
clear from the definition, this isomorphism preserves the
pair-of-pants product.

To set the stage for our discussion of ``vanishing properties'' of the
pair-of-pants product in local Floer homology, recall that the Morse
theoretic counterpart of this product is the cup product in local
Morse homology; see, e.g., \cite{Jo,Sc:thesis}. Let $F$ be a germ of a
smooth function near its isolated critical point $p\in \R^m$.  One can
show that the cup product in $\HM_*(F,p)$ is trivial unless $p$ is a
local maximum of $F$. (In the latter case, $\HM_*(F,p)$ is
concentrated in degree $m$ and $u^k=u$ for all $k$, where $u$ is the
generator of $\HM_m(F,p)=\Z_2$.) In particular, $u\cdot v=0$ for any
two distinct elements $u$ and $v$ in $\HM_*(F,p)$ regardless of
whether $p$ is a local maximum or not. These properties are inherited,
in a somewhat weaker form, by the pair-of-pants product.

\begin{Proposition}
\labell{prop:product}

Assume that $\gamma(0)$ is not a local symplectically degenerate
maximum of the germ of $\varphi_H$ at $\gamma(0)$. Then the product in
$\HF_*(H,\gamma)$ is ``nilpotent'', i.e., there exists $r_0$,
depending only on the linearized flow along $\gamma$, such that
$u_1\cdot\ldots\cdot u_r=0$ for any admissible $r\geq r_0$ and any
classes $u_1,\ldots, u_r$ in $\HF_*(H,\gamma)$.

\end{Proposition}

\begin{proof}
  The proof is essentially the observation that, unless $\gamma(0)$ is
  a local symplectically degenerate maximum, the degree $l$ of
  $u_1\cdot\ldots\cdot u_r$ is necessarily outside the support of
  $\HF_*(H^{\# r},\gamma^r)$ for a large enough $r$. Let, as above,
  $u_i\in \HF_{l_i}(H,\gamma)$ and $u_i\neq 0$.  Then the mean
  $(l_1+\ldots+l_r)/r$ also lies in the support of $\HF_*(H,\gamma)$,
  which, in turn, is contained in $(-\infty,\,\Delta_H(\gamma)+n)$,
  since $\gamma(0)$ is not a local symplectically degenerate maximum
  of the germ of $\varphi_H$. Thus,
$$
(l_1+\ldots+l_r)/r-\Delta_H(\gamma)-n\leq -\delta
$$
for some $\delta>0$ independent of $r$ and $l_1,\ldots,l_r$. It follows that
\begin{equation*}
\begin{split}
l-r\Delta_H(\gamma) &=
l_1+\ldots+l_r-(r-1)n-r\Delta_H(\gamma)\\
&=r\left(\frac{l_1+\ldots+l_r}{r}-\Delta_H(\gamma)-n\right)+n\\
&\leq -\delta r+n.
\end{split}
\end{equation*}
The support of $\HF_*(H^{\# r},\gamma^r)$ is contained in
$[r\Delta_H(\gamma)-n,\,r\Delta_H(\gamma)+n]$. Hence, when $r>
2n/\delta$, the degree $l$ of the product is outside the support.
\end{proof}

In Proposition \ref{prop:product}, the assumption that $\gamma(0)$ is
not a local symplectically degenerate maximum is essential as the
following example shows.

\begin{Example}
Assume that $p$ is a strict local maximum of an autonomous Hamiltonian
$H$ and the Hessian of $H$ at $p$ is identically zero. Then
$\HF_*(H^{\# k},p)=\HM_{*+n}(H,p)$ for every $k$ and the isomorphism
intertwines the pair-of-pants product and the cup product; cf.\
Example \ref{ex:LFH-LMH}. (This is essentially the fact that the
pair-of-pants product in Floer homology of a $C^2$-small autonomous
Hamiltonian is equal to the cup product in its Morse homology,
\cite{Sc:thesis}; see Section \ref{sec:morse-floer}.) Hence, denoting
by $u$ the generator of $\HF_n(H,p)=\Z_2$, we see that $u^k\neq 0$ for
any $k$ and, moreover, $u^k$ is the generator of $\HF_n(H^{\#
k},p)=\Z_2$.  Replacing the requirement that $d^2H_p=0$ by the
condition that the Hessian is small, we also note that the
pair-of-pants product can be non-trivial even if $p$ is
non-degenerate.
\end{Example}

A slightly more elaborate version of the argument from this example
proves that $u^k$ is a generator of $\HF_n(H^{\# k},\gamma)=\Z_2$ for
any symplectically degenerate maximum and the same is true (up to a
shift of degree) for local symplectically degenerate maxima.  (Namely,
reasoning as in Section \ref{sec:pf-lf-degen} and using Lemma
\ref{lemma:LFH-LMH}, one can equate the local Floer homology of $H$
and its iterations to the Morse homology of a generating function with
a strict, nearly degenerate maximum at $p$. Similarly to the case of
an autonomous Hamiltonian, the resulting isomorphism intertwines
products.)  This leads to a variety of characterizations of
symplectically degenerate maxima via the pair-of-pants product. For
instance, it then follows from Proposition \ref{prop:product} that
\emph{$\gamma$ is a symplectically degenerate maximum of $H$ if and
only if $\HF_n(H,p)=\Z_2$ and $u^k\neq 0$, where $u$ is the generator
of $\HF_n(H,p)$, for every admissible iteration~$k$.}

Instances where the pair-of-pants product vanishes are not exhausted
by Proposition \ref{prop:product}. For example, arguing as in the
proof of Theorem \ref{thm:persist-lf}, one can show that $u\cdot v=0$
for any two distinct elements $u$ and $v$ in $\HF_*(H,\gamma)$.

\begin{Remark}
  One may also consider products of the form $w_1\cdot \ldots \cdot
  w_r\in \HF_l(H^{\# k},\gamma^k)$ with $w_i\in \HF_{l_i}(H^{\#
  k_i},\gamma^{k_i})$, where $l=l_1+\ldots+l_r-(r-1)n$ as above and
  $k=k_1+\ldots+k_r$.  Properties of such products are more involved
  than those of the products with $k_1=\ldots=k_r=1$ considered
  above. For instance, we do not assert that Proposition
  \ref{prop:product} holds for $w_1\cdot \ldots \cdot w_r$ and it is
  certainly not true that the product of two such distinct elements
  $w_1$ and $w_2$ is necessarily zero.  However, Proposition
  \ref{prop:product} readily extends to products of this form when all
  iterations $k_i$ are bounded from above.

In conclusion note that Proposition \ref{prop:product} is analogous to
the nilpotence results for the Chas--Sullivan product established in
\cite{GH}. In fact, it is not unreasonable to expect the
corresponding local homology groups and products to be isomorphic; cf.\
\cite{AS1,AS2,SW,Vi} and references therein.
\end{Remark}

\section{Proof of Theorem \ref{thm:main}}
\labell{sec:pf-main}

The proof of Theorem \ref{thm:main} is based on the analysis of two
cases, similarly to the argument from \cite{Gi:conley} establishing the
Conley conjecture. Namely, since $\HF_n(H)\neq 0$, there exists a
one-periodic orbit $x$ of $H$ with $\HF_n(H,x)\neq 0$. Thus,
$\Delta_H(x)\geq 0$. The first, ``non-degenerate'', case is where
$\Delta_H(x)>0$, while in the second, ``degenerate'', case
$\Delta_H(x)=0$, i.e., $x$ is a symplectically degenerate
maximum. Note that since, in general, $x$ is not unique, the two cases are
not mutually exclusive for a given Hamiltonian $H$. Furthermore, we
emphasize that here, as is required in Theorem \ref{thm:main}, $M$ is
assumed to be closed and symplectically aspherical.

\subsection{Stability of Floer homology}
\labell{sec:stab}

In the proof, we will need the following simple observation asserting
that filtered Floer homology is stable, i.e., cannot be destroyed by a
relatively small perturbation of the Hamiltonian, cf.\
\cite{BC,Ch}. Let $K$ and $F$ be Hamiltonians on $M$. Set
$$
E^+=\int_0^1 \max_M F_t\,dt\quad\text{and}\quad
E^-=-\int_0^1 \min_M F_t\,dt
$$
so that $\|F\|=E^++E^-$ is the Hofer energy of $F$. Furthermore, let
$$
E^+_{0}=\max\{E^+,0\}\quad\text{and}\quad
E^-_{0}=\max\{E^-,0\}\quad\text{and}\quad
E_0(F)=E_{0}^+ +E_{0}^-.
$$
Then
\begin{itemize}
\item $\HF^{(a+E,\,b+E)}_*(K\# F)\neq 0$ for any interval $(a,\,b)$
and any non-negative constant $E\geq E_0(F)$, whenever
the natural ``quotient-inclusion'' map
$$
\kappa\colon \HF^{(a,\,b)}_*(K)\to \HF^{(a+2E,\,b+2E)}_*(K)
$$
is non-zero.
\end{itemize}
This fact is an immediate consequence of commutativity of the
following diagram
\begin{displaymath}
\xymatrix
{
\HF^{(a,\,b)}_*(K) \ar[r] \ar[dr]_{\kappa}
& \HF^{(a+E,\,b+E)}_*(K\# F) \ar[d] \\
& \HF^{(a+2E,\,b+2E)}_*(K)
}
\end{displaymath}
where the horizontal arrow is induced by the linear homotopy from $K$
to $K\# F$ and the vertical arrow is induced by the linear homotopy
from $K\# F$ to $K$; see, e.g.,~\cite{Gi:coiso}.

\begin{Remark}
  This stability result is admittedly very crude and can be refined in
  a number of ways.  For instance, as is clear from its proof, the
  intervals $(a+E,\,b+E)$ and $(a+2E,\,b+2E)$ can be replaced by the
  intervals $(a+E^+_0,\,b+E^+_0)$ and, respectively,
  $(a+E_0,\,b+E_0)$. However, the present version of stability lends
  itself conveniently for the proof of Theorem \ref{thm:main} and
  affords some notational simplifications, while a more precise
  statement appears to only result in a marginally sharper upper bound on the
  action--index gap.
\end{Remark}

\subsection{The ``non-degenerate'' case: $\HF_n(H,x)\neq 0$ and
  $\Delta_H(x)>0$}

We deal with this case under somewhat less restrictive assumptions
that $\HF_*(H,x)\neq 0$ and $\Delta_H(x)>0$.  Then, as is easy to see,
within every infinite set of admissible iterations there exists an infinite
sequence $l_1<l_2<\ldots$ such that
$$
\check{l}\leq l_{i+1}-l_i\leq \hat{l},
$$
where $\check{l}$ and $\hat{l}$ are independent of $i$ and
\begin{equation}
\labell{eq:l}
\frac{2n}{\Delta_H(x)}<\check{l}.
\end{equation}

The local Floer homology $\HF_*(H^{\# l_i},x^{l_i})$ is non-trivial
and, by (LF5), supported in the interval
$(l_i\Delta_H(x)-n,\, l_i\Delta_H(x)+n)$.  As a consequence, the groups
$\HF_*(H^{\# l_{j}},x^{l_{j}})$ and $\HF_*(H^{\# l_i},x^{l_i})$ have
disjoint support when $j\neq i$.

Adding a constant to $H$, we can assume without loss of generality
that $A_H(x)=0$, and hence $A_{H^{\# k}}(x^k)=0$ for all $k$. Set
$$
E:=\max_{r=1,\ldots, \hat{l}}E_0\big(H^{\# r}\big)
$$
and let $(a,\,b)$ be an arbitrary interval containing zero and such that
$(a+2E,\,b+2E)$ also contains zero, i.e., $a+2E<0<b$.

The sequence $\nu_j$ is picked as a subsequence of $l_i$, skipping at
most every second term. Assume that $\nu_1,\ldots, \nu_{j-1}=l_{i-1}$
and the periodic orbits $y_1,\ldots,y_{j-1}$ have been chosen. Our
goal is to find a $\nu_j$-periodic orbit $y=y_j$ with either
$\nu_j=l_i$ or $\nu_j=l_{i+1}$ satisfying the requirements of Theorem
\ref{thm:main}. (The first orbit $y_1$ and the period $\nu_1$ equal to
$l_1$ or $l_2$ are chosen in a similar fashion.)

Fix $m$ such that $\HF_m(H^{\# l_i},x^{l_i})\neq 0$. By (LF5),
$$
|m-\Delta_{H^{\# l_i}}(x^{l_i})|\leq n.
$$
Under the above assumptions, $y_j$ and $\nu_j$ are chosen differently
in each of the following three cases.

\emph{Case 1: $\HF_m^{(a,\,b)}(H^{\# l_i})=0$.} It is easy to see that
in this case $H$ has an $l_i$-periodic orbit $y$, killing the
contribution of $\HF_m(H^{\# l_i},x^{l_i})$ to $\HF_m^{(a,\,b)}(H^{\#
l_i})$, such that $|\Delta_{H^{\# l_i}}(y)-m|\leq n+1$ and the action
$A_{H^{\# l_i}}(y)\neq 0$ is in the interval $(a,\,b)$.  Set $\nu_j=l_i$
and $y_j=y$. It is clear that the action and index gaps for $x^{l_i}$
and $y$ are bounded from above by $\max\{|a|,b\}$ and, respectively,
$2n+1$ and the action gap is strictly positive.

\emph{Case 2: $\HF_m^{(a+E,\,b+E)}(H^{\# l_{i+1}})\neq 0$.}  Under
this assumption, there exists an $l_{i+1}$-periodic orbit $y$ with
action in the interval $(a+E,\,b+E)$ and $|\Delta_{H^{\#
l_{i+1}}}(y)-m|\leq n$.  We set $\nu_j=l_{i+1}$ and $y_j=y$. To verify
the requirements of the theorem, we first note that, since $A_H(x)=0$, we have
$$
|A_{H^{\#l_{i+1}}}(x^{l_{i+1}})-A_{H^{\# l_{i+1}}}(y)|
=|A_{H^{\# l_{i+1}}}(y)|\leq \max \{|a+E|,b+E\},
$$
Furthermore, as is easy to check,
$$
\check{l}\Delta_H(x)-2n\leq|\Delta_{H^{\#
l_{i+1}}}(x^{l_{i+1}})-\Delta_{H^{\# l_{i+1}}}(y)| \leq
\hat{l}\Delta_H(x)+2n.
$$
The latter inequalities give lower and upper bounds on the difference of
the mean indices and, by \eqref{eq:l}, show that this difference is
non-negative. (This is the only case where we cannot guarantee that
the action gap is strictly positive.)

\emph{Case 3: $\HF_m^{(a,\,b)}(H^{\# l_i})\neq 0$, but
$\HF_m^{(a+E,\,b+E)}(H^{\# l_{i+1}})= 0$.}
First note that
$$
H^{\# l_{i+1}}=H^{\# l_i}\# F, \quad\text{where}\quad F=H^{\# (l_{i+1}-l_i)},
\quad\text{and}\quad E\geq E_0(F).
$$
Using stability of filtered Floer homology as in Section
\ref{sec:stab} with $K=H^{\# l_i}$ and $F=H^{\# (l_{i+1}-l_i)}$, we
see that the quotient--inclusion map
$$
\kappa\colon \HF_m^{(a,\,b)}(H^{\# l_i})\to
\HF_m^{(a+2E,\,b+2E)}(H^{\# l_i})
$$
is necessarily zero, for $\HF_m^{(a+E,\,b+E)}(H^{\# l_{i+1}})=
0$. Since $\HF_m^{(a,\,b)}(H^{\# l_i})\neq 0$, we infer by a simple
exact sequence argument that $\HF_m^{(a,\,a+2E)}(H^{\# l_i})\neq 0$
or/and $\HF_{m+1}^{(b,\,b+2E)}(H^{\# l_i})\neq 0$.  In the former
case, there exists an $l_i$-periodic orbit $y$ with action in the
range $(a,\,a+2E)$ and $|m-\Delta_{H^{\# l_i}}(y)|\leq n$. In the
latter case, there exists an $l_i$-periodic orbit $y$ with action in
the range $(b,\,b+2E)$ and $|m+1-\Delta_{H^{\# l_i}}(y)|\leq n$. We
set $\nu_j=l_i$ and $y_j=y$. Then
$$
0<\min \{|a+2E|,b\}< |A_{H^{\# l_{i}}}(y)|\leq \max \{|a|,b+2E\}
$$
and
$$
|\Delta_{H^{\# l_i}}(x^{l_{i}})-\Delta_{H^{\# l_i}}(y)|\leq 2n+1.
$$

Combining the three cases above, it is immediate to see that the
constants $e$ and $\delta$ from the statement of the theorem are then
given by
$$
e=\max \{|a|,b+2E\}\quad\text{and}\quad
\delta=\max\{2n+1,2n+\hat{l}\Delta_H(x)\}
$$
and that the index gap or the action gap is necessarily positive. This
completes the proof of the theorem in the ``non-degenerate'' case.

\subsection{The ``degenerate'' case: $\HF_n(H,x)\neq 0$ and $\Delta_H(x)=0$}
This is the case where $x$ is a symplectically degenerate maximum of
$H$. By Theorem \ref{thm:persist-lf}, $x$ is strongly degenerate (thus
every $k$ is admissible for $x$) and $\HF_n(H^{\# k},x^k)\neq 0$ for all
$k\geq 1$.  Furthermore, Propositions 4.5 and 4.7 from
\cite{Gi:conley} assert that for every $\eps>0$ there exists an
integer $k_\eps>0$ such that for all $k>k_\eps$ we have
$$
\HF_{n+1}^{(kc,\,kc+\eps)} (H^{\# k})\neq 0,
$$
where $c=A_H(x)$. Hence, $\varphi$ has a $k$-periodic orbit $z_k$ with
$$
0<|A_{H^{\# k}}(x^k)-A_{H^{\# k}}(z_k)|<\eps
$$
and
$$
1\leq |\Delta_{H^{\# k}}(x^k)-\Delta_{H^{\# k}}(z_k)|\leq 2n+1.
$$
Thus, given a quasi-arithmetic sequence of admissible iterations
$l_i$, we can take as $\nu_j$ the ``tail'' of this sequence, i.e., its
subsequence formed by $l_i>k_\eps$, and set $y_j=z_{\nu_j}$.

This completes the proof of Theorem \ref{thm:main}.

\section{Persistence of isolation}
\labell{sec:persist-iso}

The main objective of this section, which is independent of the rest
of the paper, is to prove Proposition \ref{prop:persist-iso}
asserting that an isolated fixed point of a diffeomorphism remains
isolated under admissible iterations. In fact, we establish the following slightly
more general result:

\begin{Proposition}
\labell{prop:persist-iso2}

Let $p\in M$ be a fixed point of a family of $C^1$-smooth
diffeomorphisms $\varphi_s\colon M\to M$ with $s\in [0,\,1]$ and let
$k$ be an admissible iteration of $\varphi_s$ (for all $s$) with
respect to $p$. Then, for any $s$, every $k$-periodic point of
$\varphi_s$ in a sufficiently small neighborhood of $p$ (depending on
$k$, but not on $s$) is a fixed point of $\varphi_s$.

\end{Proposition}

When $\varphi_s$ is independent of $s$, i.e., $\varphi_s\equiv
\varphi$, and $p$ is isolated, this result turns into Proposition
\ref{prop:persist-iso}.  When $p$ is uniformly isolated, we obtain the
parametric version of Proposition \ref{prop:persist-iso} stated in
Remark ~\ref{sec:persist-iso}.

\begin{proof}
  Since the problem is local, we can assume without lost of generality
  that $M=\R^m$ and $p=0$.  Fixing an admissible iteration $k$, we
  need to show that every $k$-periodic point of $\varphi_s$
  sufficiently close to $p$ is a fixed point, i.e., every fixed point
  of $\varphi_s^k$ near $p$ is in fact a fixed point of $\varphi_s$.

We start with an observation of a general nature.  Let $\xi$ be a map
$\Z_k\to \R^m$.  Set $\dot{\xi}_l=\xi_{l+1}-\xi_{l}$, where
$\xi_l=\xi(l)$ and $l\in \Z_k$, and
$\|\xi\|_{L^1}=\|\xi_1\|+\ldots+\|\xi_k\|$.  Thus, $\dot{\xi}$ is
again a map $\Z_k\to \R^{m}$ and $\|\cdot\|_{L^1}$ is a norm on the
linear space of maps $\xi$. We claim that
\begin{equation}
\labell{eq:xi-dot}
\|\xi\|_{L^1}\leq c(k)\|\dot{\xi}\|_{L^1}
\text{ whenever $\xi$ has zero mean, i.e., $\xi_1+\ldots+\xi_k=0$},
\end{equation}
where the constant $c(k)$ depends only on $k$ and $m$.  Indeed,
$1/c(k)$ is the minimum of the function $\xi\mapsto
\|\dot{\xi}\|_{L^1}$ on the $\|\cdot\|_{L^1}$-unit sphere in the
linear space of all maps $\xi$ with zero mean. It is clear that this
minimum is strictly positive, and hence $c(k)$ is finite. (The choice
of the norm in \eqref{eq:xi-dot} effects only the numerical value of
$c(k)$, which is immaterial for our purposes.)

To illustrate the idea of the proof, let us first consider, as an
example, a particular case of the proposition.

\begin{Example}
\labell{exam:pc} Assume that $\varphi_s\equiv\varphi$ is independent of $s$.
Furthermore, assume that $d\varphi_p=\id$, i.e., $\varphi=\id+f$,
where $df_p=0$ and hence $\|f\|_{C^1}$ is small on a small
neighborhood of $p$. We claim that a $k$-periodic
orbit $z=\{z_1,\ldots,z_k\}$ of $\varphi$ is necessarily a fixed point of $\varphi$,
whenever $z$ is close to $p$. Indeed, $\dot{z}_l=f(z_l)$ and
$\|\ddot{z_l}\|=\|f(z_{l+1})-f(z_l)\|\leq
\|f\|_{C^1}\cdot\|\dot{z}_l\|$.  Hence,
$$
\|\ddot{z}\|_{L^1}\leq \|f\|_{C^1}\cdot\|\dot{z}\|_{L^1}.
$$
Let the neighborhood containing the orbit $z$ be so small
that $\|f\|_{C^1}<c(k)^{-1}$. Then, applying \eqref{eq:xi-dot} with
$\xi=\dot{z}$, we conclude that $\dot{z}=0$. In other words, $z$ is a
constant $k$-periodic orbit, i.e., a fixed point, of $\varphi$. (This
argument is a discrete version the Yorke period estimate, \cite{Yo};
cf.\ \cite[pp.\ 184--185]{HZ}.)
\end{Example}

The proof of the general case is essentially a combination of the
argument from Example \ref{exam:pc} and of an application of the
inverse function theorem.

First note that by compactness of $[0,\,1]$ it suffices to prove the
result for $s$ in an arbitrarily small neighborhood $I$ of $s_0\in
[0,\,1]$. If one is not an eigenvalue of $d(\varphi_{s_0})_p$, the
same is true for $d(\varphi_{s_0}^k)_p$ since $k$ is admissible, and
the assertion follows from the inverse function theorem. Thus, we
can assume that $\lambda=1$ is among the eigenvalues.  Denote by
$S_\rho$ the circle of radius $\rho>0$ centered at one.  Let
$\rho>0$ be so small that the only eigenvalue of
$d(\varphi_{s_0})_p$ within $S_\rho$ is one and, moreover, the same
is true for $d(\varphi_{s_0}^k)_p$, i.e., $\lambda^k$ is outside
$S_\rho$ for every eigenvalue $\lambda\neq 1$.

Let us decompose $\R^m$ as $V(s)\times W(s)$ so that the linearization
$d(\varphi_s)_p$ splits as the direct sum of a linear map on $V(s)$
whose eigenvalues are outside $S_\rho$ and a linear map on $W(s)$ with
all eigenvalues within $S_\rho$. Then $k$ is admissible for all
$\varphi_s$ with $s$ in a small neighborhood $I$ of $s_0$ (depending
on $\rho$), and the spaces $V(s)$ and $W(s)$ have constant dimensions
and depend smoothly on $s$. Hence, conjugating $\varphi$ by a linear
transformation (smooth in $s$), we can make $V(s)$ and $W(s)$
independent of $s$. Set $V=V(s)$ and $W=W(s)$.  The splitting of
$d(\varphi_s)_p$ gives rise to the splitting of $d(\varphi_s)_p$ and,
once $\rho>0$ and $I$ are sufficiently small, all eigenvalues of
$d(\varphi_s^k)_p\mid_W$ are within $S_\rho$ while all eigenvalues of
$d(\varphi_s^k)_p\mid_V$ are outside $S_\rho$.  In particular,
$\id-d(\varphi_s^k)_p\mid_V$ is invertible.

In what follows, we will denote $\varphi_s$ by $\varphi$
suppressing the superscript $s$ and assuming that $s$ is in a
neighborhood $I$ of $s_0$ and that $\rho>0$ and $I$ are as
small as necessary.

Let $(v_0,w_0)$ and $(v_1,w_1)$ be $k$-periodic points of $\varphi$ in
the ball $B(r)\subset V\times W$ of radius $r$, centered at the origin
$p$. Then
\begin{equation}
\labell{eq:w-v}
\|v_1-v_0\|\leq C(r)\|w_1-w_0\|,\quad\text{where $C(r)\to 0$ uniformly in $s\in I$
as $r\to 0$.}
\end{equation}
To see this, note that for any fixed point $(v,w)$ of $\varphi^k$, we
must have $v=\psi_k(v,w)$, where $\psi_k$ the $V$-component of
$\varphi^k$. The linearization of $\id-\psi_k(\cdot,0)$ at the origin
$p$ is non-degenerate, for $k$ is admissible. Thus, by the implicit
function theorem, there exists a unique smooth map $w\mapsto v(w)$ on
a neighborhood of the origin in $W$, solving the equation
$v=\psi_k(v,w)$. In particular, $v_0=v(w_0)$ and $v_1=v(w_1)$.
Furthermore, using the fact that $d\varphi^k_p$ is split, it is easy
to show that the linearization $Dv_p$ of this map at the origin $p$ is
identically zero. Hence, $C(r)=\|Dv\|_{C^0(B(r))}\to 0$ as $r\to 0$
(uniformly in $s$), and \eqref{eq:w-v} follows.

Let us set $\varphi(v,w)=(\psi_w(v),\eta_v(w))$, where $v\in V$ and
$w\in W$. Here, we view the $V$-component $\psi$ of $\varphi$ as a
family of maps $V\to V$ parametrized by $w\in W$ and, likewise, the
$W$-component $\eta$ is a family of maps $W\to W$ parametrized by $V$.

Since $d\varphi_p\mid_W$ has all eigenvalues within the
$\rho$-neighborhood of one, $d\varphi_p\mid_W$ can be made close to the
identity, up to an error of order $O(\rho)$, by conjugating $\varphi$
by a linear map depending smoothly on $s$. As a consequence, we may
assume without loss of generality that $\eta_v$ is arbitrarily
$C^1$-close to $\id$ on a small neighborhood $B_W$ of $0\in W$ for all
$v$ in some ball $B_V$ centered at $0\in V$. Setting $\eta_v=\id+f_v$,
we chose $\rho$, the interval $I$, the conjugation, and the balls
$B_W$ and $B_V$ so that
\begin{equation}
\labell{eq:max-norm}
\max_{v\in B_V}\|f_v\|_{C^1}\leq \frac{1}{2c(k)}.
\end{equation}

Let $z=(v_0,w_0)$ be a $k$-periodic point of $\varphi$ and let
$(v_l,w_l)=z_l=\varphi^l(z)$. Our next goal is to show that $w_l=w_0$
for all $l\in \Z_k$, provided that $z$ is sufficiently close to the
origin. Without loss of generality, we may assume that $z_l\in
B(r)\subset B_V\times B_W$ for all $l$. By definition,
\begin{equation*}
\begin{cases}
v_{l+1} &= \psi_{w_{l}}(v_{l}),\\
w_{l+1} &= w_{l}+ f_{v_{l}}(w_{l}).
\end{cases}
\end{equation*}
Thus, $\dot{w}_{l+1}=w_{l+1}-w_{l}=f_{v_{l}}(w_{l})$ and
\begin{equation*}
\begin{split}
\ddot{w}_{l+1}=\dot{w}_{l+1}-\dot{w}_{l}
&=f_{v_{l}}(w_{l})-f_{v_{l-1}}(w_{l-1})\\
&=f_{v_{l}}(w_{l})-f_{v_{l}}(w_{l-1})\\
&\quad+f_{v_{l}}(w_{l-1})-f_{v_{l-1}}(w_{l-1}).
\end{split}
\end{equation*}
Therefore,
$$
\|\ddot{w}_{l+1}\|\leq \|f_{v_{l}}(w_{l})-f_{v_{l}}(w_{l-1})\|
+\|f_{v_{l}}(w_{l-1})-f_{v_{l-1}}(w_{l-1})\| .
$$
Clearly,
$$
\|f_{v_{l}}(w_{l})-f_{v_{l}}(w_{l-1})\|
\leq \|f_{v_{l}}\|_{C^1}\cdot \|\dot{w}_{l}\|
$$
and, for some constant $C$ independent of $s$, we obtain using \eqref{eq:w-v} that
\begin{equation*}
\begin{split}
\|f_{v_{l}}(w_{l-1})-f_{v_{l-1}}(w_{l-1})\| &\leq
\|f_{v_{l}}-f_{v_{l-1}}\|_{C^0}\\
&\leq C\cdot \|v_{l}-v_{l-1}\|\\
&\leq C\cdot C(r)\cdot\|\dot{w}_{l}\|.
\end{split}
\end{equation*}
Combining these inequalities, we see that
$\|\ddot{w}_{l+1}\|\leq \big(\|f_{v_{l}}\|_{C^1}+C\cdot C(r)\big)
\|\dot{w}_{l}\|$, and hence
$$
\|\ddot{w}\|_{L^1}\leq \Big(\max_{v\in B_V}\|f_{v}\|_{C^1}+C\cdot C(r)\Big)
\|\dot{w}\|_{L^1}.
$$
Therefore, by \eqref{eq:max-norm}, once $r$ is so small that $C\cdot
C(r)< c(k)^{-1}/2$, we have either
$\|\ddot{w}\|_{L^1}<c(k)^{-1}\|\dot{w}\|_{L^1}$ or $\dot{w}\equiv
0$. On the other hand, \eqref{eq:xi-dot} applied to $\xi=\dot{w}$,
yields that $\|\ddot{w}\|_{L^1}\geq c(k)^{-1}\|\dot{w}\|_{L^1}$.  Thus, as
in Example \ref{exam:pc}, $\dot{w}\equiv 0$, and hence
$w_0=\ldots=w_1$.

It remains to show that $v_1=\ldots=v_k$, for then $z=(v_0,w_0)$ is a
fixed point of $\varphi$. Note that $v_1,\ldots, v_k$ is a
$k$-periodic orbit of $\psi_{w_0}$ lying in $V\times w_0$. By the
inverse function theorem, $\psi_w$ has a unique non-degenerate fixed
point $(v(w),w)$ near $(0,w)$ for every $w$ near the origin, and $k$
is an admissible iteration of $\psi_w$.  Furthermore, applying the
inverse function theorem to $\psi_w^k$, we see that every $k$-periodic
orbit of $\psi_w$ in a small neighborhood $U_w$ of $(v(w),w)$ in
$V\times w$ is the fixed point $(v(w),w)$. Clearly, the size of $U_w$
is bounded from below when $w$ is close to the origin and $s$ is close
to $s_0$. Thus, every $k$-periodic orbit of $\psi_w$ close to $0\in V$
is in fact the fixed point of $\psi_w$. In particular, $v_0$ is the
fixed point of $\psi_{w_0}$ and $z$ is a fixed point of $\varphi$.
This concludes the proof of the proposition.
\end{proof}

\begin{Remark}
\labell{rmk:SS}

Combining the proof of Proposition \ref{prop:persist-iso2} and the
proof of the Shub--Sullivan theorem (see \cite{SS}), it is easy to see
that for all admissible $k$ the index of $\varphi^k$ at $p$ is equal,
up to a sign, to the index of $\varphi$ at $p$.
\end{Remark}

\end{document}